\documentclass[a4paper]{amsart}

\usepackage{epic, eepic, amsfonts, latexsym, amssymb, graphicx,
multicol, mathrsfs, color, amscd, verbatim, paralist,
xspace, url, euscript, stmaryrd,  amsmath, enumitem,
bbold, multirow, tikz, mathtools}

\usepackage{scalerel,stackengine}
\stackMath
\newcommand\reallywidehat[1]{%
\savestack{\tmpbox}{\stretchto{%
  \scaleto{%
    \scalerel*[\widthof{\ensuremath{#1}}]{\kern-.6pt\bigwedge\kern-.6pt}%
    {\rule[-\textheight/2]{1ex}{\textheight}}%WIDTH-LIMITED BIG WEDGE
  }{\textheight}% 
}{0.5ex}}%
\stackon[1pt]{#1}{\tmpbox}%
}
\usepackage[all,pdf,cmtip]{xy}

\usepackage[colorlinks, linkcolor=blue, citecolor=magenta, urlcolor=cyan]{hyperref}

\usepackage{chngcntr}
\usepackage{apptools}

\def\tilde{\widetilde}
\def\hat{\widehat}
\renewcommand\bar{\overline}

\makeatletter
\newcommand*{\da@rightarrow}{\mathchar"0\hexnumber@\symAMSa 4B }
\newcommand*{\da@leftarrow}{\mathchar"0\hexnumber@\symAMSa 4C }
\newcommand*{\xdashrightarrow}[2][]{%
  \mathrel{%
    \mathpalette{\da@xarrow{#1}{#2}{}\da@rightarrow{\,}{}}{}%
  }%
}
\newcommand{\xdashleftarrow}[2][]{%
  \mathrel{%
    \mathpalette{\da@xarrow{#1}{#2}\da@leftarrow{}{}{\,}}{}%
  }%
}
\newcommand*{\da@xarrow}[7]{%
  \sbox0{$\ifx#7\scriptstyle\scriptscriptstyle\else\scriptstyle\fi#5#1#6\m@th$}%
  \sbox2{$\ifx#7\scriptstyle\scriptscriptstyle\else\scriptstyle\fi#5#2#6\m@th$}%
  \sbox4{$#7\dabar@\m@th$}%
  \dimen@=\wd0 %
  \ifdim\wd2 >\dimen@
    \dimen@=\wd2 %   
  \fi
  \count@=2 %
  \def\da@bars{\dabar@\dabar@}%
  \@whiledim\count@\wd4<\dimen@\do{%
    \advance\count@\@ne
    \expandafter\def\expandafter\da@bars\expandafter{%
      \da@bars
      \dabar@ 
    }%
  }%  
  \mathrel{#3}%
  \mathrel{%   
    \mathop{\da@bars}\limits
    \ifx\\#1\\%
    \else
      _{\copy0}%
    \fi
    \ifx\\#2\\%
    \else
      ^{\copy2}%
    \fi
  }%   
  \mathrel{#4}%
}
\makeatother

\def\fm{{\mathfrak m}}

\def\CC{{\mathbb C}}

\def\NN{{\mathbb N}}

\def\PP{{\mathbb P}}

\def\ra{\rightarrow}

\def\hat{\widehat}
\def\reg{\mathop{\rm reg}\nolimits}

\def\im{\mathop{\rm im}\nolimits}

\def\GL{\mathop{\rm GL}\nolimits}
\def\SL{\mathop{\rm SL}\nolimits}
\def\SU{\mathop{\rm SU}\nolimits}
\def\U{\mathop{\rm U}\nolimits}

\def\Cat{\mathop{\rm Cat}\nolimits}
\def\s{\mathop{\rm s}\nolimits}
\def\ss{\mathop{\rm ss}\nolimits}
\def\ps{\mathop{\rm ps}\nolimits}

\def\Res{\mathop{\rm Res}\nolimits}

\def\Soc{\mathop{\rm Soc}\nolimits}

\def\Jac{\mathop{\rm Jac}\nolimits}

\def\Spec{\mathop{\rm Spec}\nolimits}
\def\Proj{\mathop{\rm Proj}\nolimits}

\def\Gr{\mathop{\rm Gr}\nolimits}

\def\gitq{/\hspace{-0.1cm}/}

\def\Hess{\mathop{\rm Hess}\nolimits}

\newcommand{\bbf}{\mathfrak{f}}

\newcommand\co{\colon} 

\newtheorem{theorem}{THEOREM}[section]
\newtheorem{corollary}[theorem]{Corollary}

\newtheorem{proposition}[theorem]{Proposition}
\newtheorem{conjecture}[theorem]{Conjecture}

\theoremstyle{definition}
\newtheorem{example}[theorem]{Example}

\newtheorem{prop}[theorem]{Proposition}

\newtheorem{claim}[theorem]{Claim}
\newtheorem{lemma}[theorem]{Lemma}
\newtheorem{openprob}[theorem]{Open Problem}

\theoremstyle{remark}
\newtheorem{remark}[theorem]{Remark}

\makeatletter
\def\blfootnote{\xdef\@thefnmark{}\@footnotetext}
\makeatother

\begin{document}

\title[Associated forms: current progress and open problems]{Associated forms:
\vspace{0.1cm}\\ 
current progress and open problems}

\author[Isaev]{Alexander Isaev}

\address{Mathematical Sciences Institute\\
Australian National University\\
Canberra, Acton, ACT 2601, Australia}
\email{alexander.isaev@anu.edu.au}

\maketitle

\thispagestyle{empty}

\pagestyle{myheadings}

\begin{abstract}  
Let $d\ge 3$, $n\ge 2$. The object of our study is the morphism $\Phi$, introduced in earlier articles by J. Alper, M. Eastwood and the author, that assigns to every homogeneous form of degree $d$ on $\CC^n$ for which the discriminant $\Delta$ does not vanish a form of degree $n(d-2)$ on the dual space called the associated form. This morphism is $\SL_n$-equivariant and is of interest in connection with the well-known Mather-Yau theorem, specifically, with the problem of explicit reconstruction of an isolated hypersurface singularity from its Tjurina algebra. Letting $p$ be the smallest integer such that the product $\Delta^p\Phi$ extends to the entire affine space of degree $d$ forms, one observes that the extended map defines a contravariant. In the present paper we survey known results on the morphism $\Phi$, as well as the contravariant $\Delta^p\Phi$, and state several open problems. Our goal is to draw the attention of complex analysts and geometers to the concept of the associated form and the intriguing connection between complex singularity theory and invariant theory revealed through it.
\end{abstract}

\section{Introduction}\label{intro}
\setcounter{equation}{0}

In this paper we discuss a curious connection between complex singularity theory and classical invariant theory proposed in \cite{EI} and further explored in \cite{AI1}, \cite{AI2}, \cite{AIK}, \cite{F}, \cite{FI}. What follows is a survey of known results and open problems. It is written as a substantially extended version of our recent paper \cite{I3} and is intended mainly for complex analysts and geometers. Thus, some of our expositional and notational choices may not be up to the taste of a reader with background in algebra, for which we apologize.

Consider the vector space $\CC[z_1,\dots,z_n]_d$ of homogeneous forms of degree $d$ on $\CC^n$, where $n\ge 2$, $d\ge 3$. Fix $f\in\CC[z_1,\dots,z_n]_d$ and look at the hypersurface $V_f:=\{z\in\CC^n: f(z)=0\}$. We will be interested in the situation when the singularity of $f$ at the origin is isolated, or, equivalently, when the discriminant $\Delta(f)$ of $f$ does not vanish. In this case, define $M_f:=\CC[[z_1,\dots,z_n]]/(f_{z_{{}_1}},\dots,f_{z_{{}_n}})$ to be the {\it Milnor algebra}\, of the singularity. By the Mather-Yau theorem (see \cite{MY} and also \cite{Be}, \cite{Sh}, \cite[Theorem 2.26]{GLS}, \cite{GP}), the isomorphism class of $M_f$ determines the germ of the hypersurface $V_f$ at the origin up to biholomorphism, hence the form $f$ up to linear equivalence.

In fact, for a general isolated hypersurface singularity in $\CC^n$ defined by (the germ of) a holomorphic function $F$, the Mather-Yau theorem states that, remarkably, the singularity is determined, up to biholomorphism, by $n$ and the isomorphism class of the {\it Tjurina algebra}\, $T_F:=\CC[[z_1,\dots,z_n]]/(F,F_{z_{{}_1}},\dots,F_{z_{{}_n}})$. The proof of the Mather-Yau theorem is not constructive, and it is an important open problem---called {\it the reconstruction problem}---to understand explicitly how the singularity is encoded in the corresponding Tjurina algebra. In this paper we concentrate on the homogeneous case as set out in the previous paragraph (notice that $T_f=M_f$). In this situation, the reconstruction problem was solved in \cite{IK}, where we proposed a simple algorithm for extracting the linear equivalence class of the form $f$ from the isomorphism class of $M_f$. An alternative (invariant-theoretic) approach to the reconstruction problem---which applies to the more general class of quasihomogeneous isolated hypersurface singularities---was initiated in article \cite{EI}, where we proposed a method for extracting certain numerical invariants of the singularity from its Milnor algebra (see \cite{I2} for a comparison of the two techniques). Already in the case of homogeneous singularities this approach leads to a curious concept that deserves attention regardless of the reconstruction problem and that is interesting from the purely classical invariant theory viewpoint. This concept is the focus of the present paper.

We will now briefly describe the idea behind it with details postponed until Section \ref{setup}. Let $\fm$ be the (unique) maximal ideal of $M_f$ and $\Soc(M_f)$ the socle of $M_f$, defined as $\Soc(M_f):=\{x\in M_f: x\,\fm=0\}$. It turns out that $M_f$ is a Gorenstein algebra, i.e., $\dim_{\CC}\Soc(M_f)=1$, and, moreover, that $\Soc(M_f)$ is spanned by the image $\reallywidehat{\Hess(f)}$ of the Hessian $\Hess(f)$ of $f$ in $M_f$. Observing that $\Hess(f)$ has degree $n(d-2)$, one can then introduce a form defined on the $n$-dimensional quotient $\fm/\fm^2$ with values in $\Soc(M_f)$ as follows:
$$
\fm/\fm^2 	 \to \Soc(M_f), \quad x  \mapsto y^{\,n(d-2)},
$$
where $y$ is any element of ${\mathfrak m}$ that projects to $x\in{\mathfrak m}/{\mathfrak m}^2$.  There is a canonical isomorphism ${\mathfrak m}/{\mathfrak m}^2\cong \CC^{n*}$ and, since $\reallywidehat{\Hess(f)}$ spans the socle, there is also a canonical isomorphism $\Soc(M_f) \cong \CC$. Hence, one obtains a form ${\mathbf f}$ of degree $n(d-2)$ on  $\CC^{n*}$ (i.e., an element of $\CC[e_1,\dots,e_n]_{n(d-2)}$, where $e_1,\dots,e_n$ is the standard basis of $\CC^n$), called the {\it associated form}\, of $f$.

The principal object of our study is the morphism
$$
\Phi:X_n^d\to \CC[e_1,\dots,e_n]_{n(d-2)},\quad f\mapsto{\mathbf f}
$$
of affine algebraic varieties, where $X_n^d$ is the variety of forms in $\CC[z_1,\dots,z_n]_d$ with nonzero discriminant. This map has a $\GL_n$-equivariance property (see Proposition \ref{equivariance}), and one of the reasons for our interest in $\Phi$ is the following intriguing conjecture proposed in \cite{EI}, \cite{AI1}:

\begin{conjecture}\label{conj2} For every regular $\GL_n$-invariant function $S$ on $X_n^d$ there exists a rational $\GL_n$-invariant function $R$ on $\CC[e_1,\dots,e_n]_{n(d-2)}$ defined at all points of the set\, $\Phi(X_n^d)\subset \CC[e_1,\dots,e_n]_{n(d-2)}$ such that $R\circ\Phi=S$.
\end{conjecture}

Observe that, if settled, Conjecture \ref{conj2} would imply an invariant-theoretic solution to the reconstruction problem in the homogeneous case. Indeed, on the one hand, it is well-known that the regular $\GL_n$-invariant functions on $X_n^d$ separate the $\GL_n$-orbits, and, on the other hand, the result of the evaluation of any rational $\GL_n$-invariant function at the associated form ${\mathbf f}$ depends only on the isomorphism class of $M_f$. Thus, the conjecture would yield a complete system of biholomorphic invariants of homogeneous isolated hypersurface singularities constructed from the algebra $M_f$ alone. So far, Conjecture \ref{conj2} has been confirmed for binary forms (see \cite{EI}, \cite{AI2}), and its weaker variant (which does not require that the function $R$ be defined on the entire image of $\Phi$) has been established for all $n$ and $d$ (see \cite{AI1}).

The conjecture is also rather interesting from the purely invariant-theoretic point of view. Indeed, if settled, it would imply that the invariant theory of forms in $\CC[z_1,\dots,z_n]_d$ can be extracted, by way of the morphism $\Phi$, from that of forms in $\CC[e_1,\dots,e_n]_{n(d-2)}$ at least at the level of rational invariant functions, or {\it absolute invariants}. Indeed, every absolute invariant of forms in $\CC[z_1,\dots,z_n]_d$ can be represented as the ratio of two $\GL_n$-invariant regular functions on $X_n^d$ (see \cite[Corollary 5.24 and Proposition 6.2]{Mu}).

The goal of the present survey is to draw the attention of the complex-analytic audience to the concept of the associated form and the curious connection between complex geometry and invariant theory manifested through it. In the paper, we focus on two groups of problems concerning associated forms. The first one is related to establishing Conjecture \ref{conj2} and is discussed in Sections \ref{results} and \ref{S:binaryquarticternarycubics}. The other one is also relevant to classical invariant theory but in a different way. Namely, letting $p$ be the smallest positive integer such that the product $\Delta^p\Phi$ extends to a morphism from $\CC[z_1,\dots,z_n]_d$ to $\CC[e_1,\dots,e_n]_{n(d-2)}$, by utilizing the equivariance property of $\Phi$ one observes that this product defines a contravariant of degree $np(d-1)^{n-1}-n$ of forms in $\CC[z_1,\dots,z_n]_d$. While it can be expressed via known contravariants for small values of $n$ and $d$ (see \cite{AIK} and Subsection \ref{contravarsmallnd} below), it appears to be new in general (cf.~\cite{D2}). We discuss this contravariant in Section \ref{S:contravariant} focussing on the problem of estimating the\linebreak integer $p$. Note that some of the details included in the survey have not been previously published.

For simplicity we have chosen to work over the field $\CC$ although everything that follows applies verbatim to any algebraically closed field of characteristic zero, and many of the results do not in fact require algebraic closedness. We also note that all algebraic geometry in the paper is done for complex varieties (i.e., reduced separated schemes of finite type over $\CC$) hence in the proofs it suffices to argue at the level of closed points, and this is what we do. In particular, when we speak about affine (resp.~projective) varieties, we only deal with the maximal spectra (resp.~maximal projective spectra) of the corresponding rings.

{\bf Acknowledgement.} We are grateful to M. Fedorchuk for many very helpful discussions. 

\section{Preliminaries on associated forms}\label{setup}
\setcounter{equation}{0}

In this section we provide an introduction to associated forms and their properties.

\subsection{The associated form of a nondegenerate form} 
For any finite collection of symbols $t_1,\dots,t_m$ we denote by $\CC[t_1,\dots,t_m]$ the algebra of polynomials in these symbols with complex coefficients and by $\CC[t_1,\dots,t_m]_k\subset\CC[t_1,\dots,t_m]$ the vector space of homogeneous forms in $t_1,\dots,t_m$ of degree $k\ge 0$. 
%The dimension of $\CC[t_1,\dots,t_m]_k$ is given by the well-known formula
%$$
%\dim_{\CC}\CC[t_1,\dots,t_m]_k=\left(
%\begin{array}{c}
%k+m-1\\
%k
%\end{array}
%\right).\label{dimform}
%$$
Clearly, we have
$$
\CC[t_1,\dots,t_m]=\bigoplus_{k=0}^{\infty}\CC[t_1,\dots,t_m]_k.
$$

We now fix $n\ge 2$ and let $e_1,\dots,e_n$ be the standard basis of $\CC^n$. The group $\GL_n:=\GL_n(\CC)$ (hence the group $\SL_n:=\SL_n(\CC)$) acts on $\CC^n$ via   
$$
(e_1,\dots,e_n)\mapsto (e_1,\dots,e_n)C,\,C\in\GL_n,
$$
or, equivalently, as
\begin{equation}
Cz=C(z_1,\dots,z_n):=(z_1,\dots,z_n)C^T,\,\,z=(z_1,\dots,z_n)\in\CC^n,\,\,C\in\GL_n.\label{actiononcn}
\end{equation}
This action induces an action on the space $\CC[e_1,\dots,e_n]_k$:
\begin{equation}
(CF)(e_1,\dots,e_n):=F((e_1,\dots,e_n)C),\,F\in \CC[e_1,\dots,e_n]_k,\,C\in\GL_n.\label{dualactionfe}
\end{equation}

Next, let us think of the coordinates $z_1,\dots,z_n$ on $\CC^n$ with respect to the basis $e_1,\dots,e_n$ as the elements of the basis of $\CC^{n*}$ dual to $e_1,\dots,e_n$. Then the dual action of $\GL_n$ on $\CC^{n*}$ is given by
$$
(z_1,\dots,z_n)\mapsto (z_1,\dots,z_n)C^{-T},\,C\in\GL_n.
$$
Equivalently, if we identify a point $z^*\in\CC^{n*}$ with its coordinate vector $(z_1^*,\dots,z_n^*)$ with respect to the basis $z_1,\dots,z_n$, this action is written as  
\begin{equation}
Cz^*=C(z_1^*,\dots,z_n^*)=(z_1^*,\dots,z_n^*)C^{-1},\,z^*=(z_1^*,\dots,z_n^*)\in\CC^{n*},\,C\in\GL_n.\label{actiononcn*}
\end{equation}
It leads to an action on $\CC[z_1,\dots,z_n]_k$:
\begin{equation}
(Cf)(z_1,\dots,z_n):=f\left((z_1,\dots,z_n)C^{-T}\right),\,f\in \CC[z_1,\dots,z_n]_k,\,C\in\GL_n.\label{actionfz}
\end{equation}
Two forms in either $\CC[e_1,\dots,e_n]_k$ or $\CC[z_1,\dots,z_n]_k$ that lie in the same $\GL_n$-orbit are called {\it linearly equivalent}. 

Clearly, every element of $\CC[z_1,\dots,z_n]_k$ can be thought of as a function on $\CC^n$, so to every nonzero $f\in\CC[z_1,\dots,z_n]_k$ we associate the hypersurface
$$
V_f:=\{z\in\CC^n: f(z)=0\}
$$
and consider it as a complex space with the structure sheaf induced by $f$. The singular set of $V_f$ is then the critical set of $f$. In particular, if $k\ge 2$ the hypersurface $V_f$ has a singularity at the origin. We are interested in the situation when this singularity is isolated, or, equivalently, when $V_f$ is smooth away from 0. This occurs if and only if $f$ is {\it nondegenerate}, i.e., $\Delta(f)\ne 0$, where $\Delta$ is the {\it discriminant}\, (see \cite[Chapter 13]{GKZ}). 

Fix $d\ge 3$ and define
$$
X^d_n:=\{f\in\CC[z_1,\dots,z_n]_d: \Delta(f)\ne 0\}.
$$
Observe that $\GL_n$ acts on the affine variety $X_n^d$ and note that every $f\in X_n^d$ is {\it stable}\, with respect to this action, i.e., the orbit of $f$ is closed in $X_n^d$ and has dimension $n^2$  (see \cite[Proposition 4.2]{MFK}, \cite[Corollary 5.24, Lemma 5.40]{Mu} and cf.~Subsection \ref{reviewGIT} below). It then follows by standard Geometric Invariant Theory arguments (see, e.g., \cite[Proposition 3.1]{EI}) that regular invariant functions on $X_n^d$ separate the $\GL_n$-orbits. As explained in the introduction, this is one of the facts that link Conjecture \ref{conj2} with the reconstruction problem arising from the Mather-Yau theorem.  

Fix $f\in X^d_n$ and consider the {\it Milnor algebra}\, of the singularity\ of $V_f$, which is the complex local algebra
$$
M_f:=\CC[[z_1,\dots,z_n]]/(f_{z_{{}_1}},\dots,f_{z_{{}_n}}),
$$
where $\CC[[z_1,\dots,z_n]]$ is the algebra of formal power series in $z_1,\dots,z_n$ with complex coefficients. Since the singularity of $V_f$ is isolated, it follows from the Nullstellensatz %(see, e.g., \cite[Proposition 1.72]{GLS}) 
that the algebra $M_f$ is Artinian, i.e., $\dim_{\CC}M_f<\infty$. Therefore, $f_{z_{{}_1}},\dots,f_{z_{{}_n}}$ is a system of parameters in $\CC[[z_1,\dots,z_n]]$, and, since $\CC[[z_1,\dots,z_n]]$ is a regular local ring, $f_{z_{{}_1}},\dots,f_{z_{{}_n}}$ is a regular sequence in $\CC[[z_1,\dots,z_n]]$. This yields that $M_f$ is a complete intersection (see \cite[\S 21]{Ma}).

It is convenient to utilize another realization of the Milnor algebra. Namely, it is easy to see that $M_f$ is isomorphic to the algebra $\CC[z_1,\dots,z_n]/(f_{z_{{}_1}},\dots,f_{z_{{}_n}})$, so we write  
$$
M_f=\CC[z_1,\dots,z_n]/(f_{z_{{}_1}},\dots,f_{z_{{}_n}}).
$$
Let ${\mathfrak m}$ denote the maximal ideal of $M_f$, which consists of all elements represented by polynomials in $\CC[z_1,\dots,z_n]$ vanishing at the origin. By Nakayama's lemma, the maximal ideal is nilpotent and we let $\nu:=\max\{\eta\in\NN: {\mathfrak m}^{\eta}\ne 0\}$ be the socle degree of $M_f$.

Since $M_f$ is a complete intersection, by \cite{Ba} it is a {\it Gorenstein algebra}. This means that the {\it socle}\, of $M_f$, defined as
$$
\Soc(M_f):=\{x\in M_f : x\,{\mathfrak m}=0\},
$$
is a one-dimensional vector space over $\CC$  (see, e.g., \cite[Theorem 5.3]{Hu}). We then have $\Soc(M_f)={\mathfrak m}^{\nu}$. Furthermore, $\Soc(M_f)$ is spanned by the projection $\reallywidehat{\Hess(f)}$ to $M_f$ of the Hessian $\Hess(f)$ of $f$ (see, e.g., \cite[Lemma 3.3]{Sai}). Since $\Hess(f)$ has degree $n(d-2)$, it follows that $\nu=n(d-2)$. Thus, the subspace 
\begin{equation}
\begin{array}{l}
W_f:=\CC[z_1,\dots,z_n]_{n(d-2)-(d-1)}f_{z_{{}_1}}+\dots+\\
\vspace{-0.3cm}\\
\hspace{3cm}\CC[z_1,\dots,z_n]_{n(d-2)-(d-1)}f_{z_{{}_n}}\subset\CC[z_1,\dots,z_n]_{n(d-2)}\end{array}\label{subspace}
\end{equation}
has codimension 1, with the line spanned by $\Hess(f)$ being complementary to it.

Denote by $\omega_f \co \Soc(M_f)\ra\CC$ the linear isomorphism given by the condition $\omega_f(\reallywidehat{\Hess(f)})=1$. Define a form ${\mathbf f}$ on $\CC^{n*}$ as follows. Fix $z^*\in\CC^{n*}$, let, as before, $z_1^*,\dots,z_n^*$ be the coordinates of $z^*$ with respect to the basis $z_1,\dots,z_n$, and set
\begin{equation}
{\mathbf f}(z^*):=\omega_f\left((z_1^*\widehat{z}_1+\dots+z_n^*\widehat{z}_n)^{n(d-2)}\right),\label{assocformdef}
\end{equation}
where $\widehat{z}_j$ is the projection to $M_f$ of the coordinate function $z_j\in\CC[z_1,\dots,z_n]$.

Notice that if  $i_1,\dots,i_n$ are nonnegative integers such that $i_1+\dots+i_n=n(d-2)$, the product $\widehat{z}_1^{i_1}\cdots \widehat{z}_n^{i_n}$ lies in $\Soc(M_f)$, hence we have 
\begin{equation}
\widehat{z}_1^{i_1}\cdots \widehat{z}_n^{i_n}=\mu_{i_1,\dots,i_n}(f) \reallywidehat{\Hess(f)}\label{assocformexpppp}
\end{equation}
for some $\mu_{i_1,\dots,i_n}(f)\in\CC$. In terms of the coefficients $\mu_{i_1,\dots,i_n}(f)$ the form ${\mathbf f}$ is written as 
\begin{equation}
{\mathbf f}(z^*)=\sum_{i_1+\cdots+i_n=n(d-2)}\frac{(n(d-2))!}{i_1!\cdots i_n!}\mu_{i_1,\dots,i_n}(f)
z_1^{* i_1}\cdots z_n^{* i_n}.\label{assocformexpp}
\end{equation}
One can view the expression in the right-hand side of (\ref{assocformexpp}) as an element of $\CC[z_1^*,\dots,z_n^*]_{n(d-2)}$, where we regard $z_1^*,\dots,z_n^*$ as the basis of $\CC^{n**}$ dual to the basis $z_1,\dots,z_n$ of $\CC^{n*}$. Identifying $z^*_j\in\CC^{n**}$ with $e_j\in\CC^n$, we will think of ${\mathbf f}$ as the following element of $\CC[e_1,\dots,e_n]_{n(d-2)}$:
\begin{equation}
{\mathbf f}(e_1,\dots,e_n)=\sum_{i_1+\cdots+i_n=n(d-2)}\frac{(n(d-2))!}{i_1!\cdots i_n!}\mu_{i_1,\dots,i_n}(f)
e_1^{i_1}\cdots e_n^{i_n}.\label{assocformexpp1}
\end{equation}
We call ${\mathbf f}$ given by expression (\ref{assocformexpp1}) the {\it associated form}\, of $f$.

\begin{example}\label{E:example1} \rm If $f = a_1 z_1^d + \cdots + a_n z_n^d$ for nonzero $a_i \in \CC$, then one computes $\Hess(f) = a_1 \cdots a_n(d(d-1))^n (z_1 \cdots z_n)^{d-2}$ and
$$
{\mathbf f}(e_1,\dots,e_n) = \frac{1}{a_1 \cdots a_n} \frac{(n(d-2))!}{(d!)^n} e_1^{d-2} \cdots e_n^{d-2}.
$$
\end{example}

\noindent More examples of calculating associated forms will be given in Section \ref{S:binaryquarticternarycubics}. 

It is not hard to show that each $\mu_{i_1,\dots,i_n}$ is a regular function on the affine variety $X_n^d$ (see, e.g., \cite[Proposition 2.1]{I3}). Hence, we have 
\begin{equation}
\mu_{i_1,\dots,i_n}=\frac{P_{i_1,\dots,i_n}}{\Delta^{p_{i_1,\dots,i_n}}}\label{formulaformus}
\end{equation}
for some $P_{i_1,\dots,i_n}\in\CC[\CC[z_1,\dots,z_n]_d]$ and nonnegative integer $p_{i_1,\dots,i_n}$. Here and in what follows for any affine variety $X$ over $\CC$ we denote by $\CC[X]$ its coordinate ring, which coincides with the ring ${\mathcal O}_X(X)$ of all regular functions on $X$. For example, $\CC[z_1,\dots,z_n]=\CC[\CC^n]$ and $\CC[z_1^*,\dots,z_n^*]=\CC[\CC^{n*}]$. 
%To obtain representation (\ref{formulaformus}), recall that for any irreducible affine variety $X$ and a regular nonzero function $h$ on $X$, any regular function on the affine open subset $\{x\in X: h(x)\ne 0\}$ is the ratio of a regular function on $X$ and a nonnegative power of $h$ (see, e.g., \cite[Proposition 1.40]{GW}). 

Thus, we have arrived at the morphism
$$
\Phi \co X_n^d\ra \CC[e_1,\dots,e_n]_{n(d-2)},\quad f\mapsto {\mathbf f}
$$
of affine algebraic varieties. Notice that by Example \ref{E:example1} this morphism is not injective. 

Next, recall that for any $k\ge 0$ the {\it polar pairing}\, between the spaces $\CC[z_1,\dots,z_n]_{k}$ and $\CC[e_1,\dots,e_n]_{k}$ is given as follows:\begin{equation}
\begin{array}{l}
\CC[z_1,\dots,z_n]_{k}\times\CC[e_1,\dots,e_n]_{k}\to\CC,\\
\vspace{-0.1cm}\\
(g(z_1,\dots,z_n),F(e_1,\dots,e_n))\mapsto g\diamond F:=\\
\vspace{-0.3cm}\\
\hspace{5cm}g\left(\partial/\partial e_1,\dots,\partial/\partial e_n\right) F(e_1,\dots,e_n).
\end{array}\label{polarpairing}
\end{equation}
This pairing is nondegenerate and therefore yields a canonical identification between $\CC[e_1,\dots,e_n]_{k}$ and $\CC[z_1,\dots,z_n]_{k}^*$ (see, e.g., \cite[Subsection 1.1.1]{D1} for details). Using this identification, one may regard the associated form as an element of $\CC[z_1,\dots,z_n]_{n(d-2)}^*$, in which case the morphism $\Phi$ turns into a morphism from $X_n^d$ to $\CC[z_1,\dots,z_n]_{n(d-2)}^*$; we denote it by $\tilde\Phi$.  

The morphism $\tilde\Phi$ admits a rather simple description. For $f\in X_n^d$, let $\tilde\omega_f$ be the element of $\CC[z_1,\dots,z_n]_{n(d-2)}^*$ such that: 
\begin{itemize}

\item[(i)] $\ker\tilde\omega_f=W_f$ with $W_f$ introduced in (\ref{subspace}), and 

\vspace{0.1cm}
\item[(ii)]  $\tilde\omega_f(\Hess(f))=1$. 

\end{itemize}

\noindent Clearly, $\mu_{i_1,\dots,i_n}(f)=\tilde\omega_f(z_1^{i_1}\cdots z_n^{i_n})$ for $i_1+\dots+i_n=n(d-2)$. A straightforward calculation yields:

\begin{proposition}\label{newmap}
The morphism
$$
\tilde\Phi: X_n^d\to \CC[z_1,\dots,z_n]_{n(d-2)}^*
$$
sends a form $f$ to $(n(d-2))!\,\tilde\omega_f$.
 \end{proposition}

The maps $\Phi$ and $\tilde\Phi$ are rather natural; in particular, \cite[Proposition 2.1]{AI1} implies equivariance properties for them. Recall that the actions of $\GL_n$ on $\CC[z_1,\dots,z_n]_d$, $\CC[e_1,\dots,e_n]_{n(d-2)}$ are given by formulas (\ref{actionfz}), (\ref{dualactionfe}), respectively, 
and on the dual space $\CC[z_1,\dots,z_n]_{n(d-2)}^*$ by
\begin{equation}
\begin{array}{l}
(Ch)(g):=h(C^{-1}g),\\
\vspace{-0.3cm}\\
\hspace{2cm}h\in \CC[z_1,\dots,z_n]_{n(d-2)}^*,\, g\in \CC[z_1,\dots,z_n]_{n(d-2)},\,C\in\GL_n.
\end{array}\label{dualdualhg}
\end{equation}
The isomorphism $\CC[e_1,\dots,e_n]_{n(d-2)}\simeq  \CC[z_1,\dots,z_n]_{n(d-2)}^*$ induced by the polar pairing is equivariant with respect to actions (\ref{dualactionfe}) and (\ref{dualdualhg}). 

We will now state the equivariance properties of $\Phi$ and $\tilde\Phi$:
  
\begin{proposition}\label{equivariance} For every $f\in X_n^d$ and $C\in\GL_n$ one has
\begin{equation}
\Phi(Cf)=(\det C)^2\,\Bigl(C\Phi(f)\Bigr)\,\,\hbox{and}\,\,\,\tilde\Phi(Cf)=(\det C)^2\,\Bigl(C\tilde\Phi(f)\Bigr).\label{equivarphitildephi}
\end{equation}
In particular, the morphisms $\Phi$, $\tilde\Phi$ are $\SL_n$-equivariant. 
\end{proposition}

Note that the associated form of $f\in X_n^d$ arises from the following invariantly defined map
$$
\fm/\fm^2 \to \Soc(M_f),\quad x \mapsto y^{n(d-2)},\label{coordinatefree}
$$
with $y\in\fm$ being any element that projects to $x\in\fm/\fm^2$. Indeed, ${\mathbf f}$ is derived from this map by identifying the target with $\CC$ via $\omega_f$ and the source with $\CC^{n*}$ by mapping the image of $\widehat{z}_j$ in $\fm/\fm^2$ to the element $z_j$ of the basis $z_1,\dots,z_n$ of $\CC^{n*}$. It then follows that for any rational $\GL_n$-invariant function $R$ on $\CC[e_1,\dots,e_n]_{n(d-2)}$ the value $R({\mathbf f})$ depends only on the isomorphism class of the algebra $M_f$. As stated in the introduction, this is another fact that links Conjecture \ref{conj2} with the reconstruction problem. 

\subsection{The associated form of a finite morphism} As before, let $n\ge 2$ and $d\ge 3$. We will now generalize the above construction from forms $f\in X_n^d$ to finite morphisms ${\mathfrak f}=(f_1,\dots,f_n):\CC^n\to\CC^n$ defined by $n$ forms of degree $d-1$. 

Consider the vector space $(\CC[z_1,\dots,z_n]_{d-1})^{\oplus n}$ of $n$-tuples ${\mathfrak f}=(f_1, \ldots, f_n)$ of forms of degree $d-1$.  Recall that the {\it resultant}\, $\Res$ on the space $(\CC[z_1,\dots,z_n]_{d-1})^{\oplus n}$ is a form with the property that $\Res({\mathfrak f}) \neq 0$ if and only if $f_1, \ldots, f_n$ have no common zeroes away from the origin (see, e.g., \cite[Chapter 13]{GKZ}). For an element ${\mathfrak f} = (f_1, \ldots, f_n) \in (\CC[z_1,\dots,z_n]_{d-1})^{\oplus n}$, we now introduce the algebra
$$
M_{\mathfrak f} := \CC[z_1,\dots,z_n]/ (f_1, \ldots, f_n)
$$
and recall a well-known lemma (see, e.g., \cite[Lemma 2.4]{AI2} and \cite[p.~187]{SS}):

\begin{lemma}\label{fourconds} \it The following statements are equivalent:
\begin{enumerate}
	\item[\rm (1)] the resultant $\Res(\bbf)$ is nonzero;
	\item[\rm (2)] the algebra $M_\bbf$ has finite vector space dimension;
	\item[\rm (3)] the morphism $\bbf \co \CC^n \to \CC^n$ is finite;
	\item[\rm (4)] the $n$-tuple $\bbf$ is a homogeneous system of parameters of $\CC[z_1,\dots,z_n]$, i.e., the Krull dimension of $M_\bbf$ is $0$.
\end{enumerate}
If the above conditions are satisfied, then $M_\bbf$ is a local complete intersection {\rm (}hence Gorenstein{\rm )} algebra whose socle $\Soc(M_{\mathfrak f})$ is generated in degree $n(d-2)$ by the image $\widehat{\Jac(\bbf)}$ in $M_\bbf$ of the Jacobian $\Jac(\bbf)$ of\, $\bbf$.\end{lemma}

\noindent In the above lemma
$
\Soc(M_\bbf):=\{x\in M_\bbf : x\,{\mathfrak m}=0\},
$
where the (unique) maximal ideal ${\mathfrak m}$ of $M_\bbf$ consists of all elements represented by polynomials in $\CC[z_1,\dots,z_n]$ vanishing at the origin.

Next, let $Y_n^{d-1}$ be the affine open subset of $(\CC[z_1,\dots,z_n]_{d-1})^{\oplus n}$ of all $n$-tuples of forms with nonzero resultant. Lemma \ref{fourconds} implies that for $\bbf\in Y_n^{d-1}$ the subspace 
\begin{equation}
\begin{array}{l}
W_\bbf:=\CC[z_1,\dots,z_n]_{n(d-2)-(d-1)}f_1+\dots+\\
\vspace{-0.3cm}\\
\hspace{3cm}\CC[z_1,\dots,z_n]_{n(d-2)-(d-1)}f_n\subset\CC[z_1,\dots,z_n]_{n(d-2)}
\end{array}\label{subspace1}
\end{equation}
has codimension 1, with the line spanned by $\Jac(\bbf)$ being complementary to it.

Fix $\bbf\in Y_n^{d-1}$ and denote by $\omega_\bbf \co \Soc(M_\bbf)\ra\CC$ the linear isomorphism given by the condition $\omega_\bbf(\widehat{\Jac(\bbf)})=1$. Define a form ${\mathbf f}$ on $\CC^{n*}$ as follows. Fix $z^*\in\CC^{n*}$, let, as before, $z_1^*,\dots,z_n^*$ be the coordinates of $z^*$ with respect to the basis $z_1,\dots,z_n$, and set
$$
{\mathbf f}(z^*):=\omega_\bbf\left((z_1^*\widehat{z}_1+\dots+z_n^*\widehat{z}_n)^{n(d-2)}\right),\label{assocformdef1}
$$
where $\widehat{z}_j$ is the projection to $M_\bbf$ of the coordinate function $z_j\in\CC[z_1,\dots,z_n]$.

If  $i_1,\dots,i_n$ are nonnegative integers such that $i_1+\dots+i_n=n(d-2)$, the product $\widehat{z}_1^{i_1}\cdots \widehat{z}_n^{i_n}$ lies in $\Soc(M_\bbf)$, hence we have 
$$
\widehat{z}_1^{i_1}\cdots \widehat{z}_n^{i_n}=\mu_{i_1,\dots,i_n}(\bbf) \widehat{\Jac(\bbf)}\label{assocformexpppp1}
$$
for some $\mu_{i_1,\dots,i_n}(\bbf)\in\CC$. In terms of the coefficients $\mu_{i_1,\dots,i_n}(\bbf)$ the form ${\mathbf f}$ is written as 
\begin{equation}
{\mathbf f}(z^*)=\sum_{i_1+\cdots+i_n=n(d-2)}\frac{(n(d-2))!}{i_1!\cdots i_n!}\mu_{i_1,\dots,i_n}(\bbf)
z_1^{* i_1}\cdots z_n^{* i_n}.\label{assocformexpp2}
\end{equation}
One can view the expression in the right-hand side of (\ref{assocformexpp2}) as an element of $\CC[z_1^*,\dots,z_n^*]_{n(d-2)}$, where we regard $z_1^*,\dots,z_n^*$ as the basis of $\CC^{n**}$ dual to the basis $z_1,\dots,z_n$ of $\CC^{n*}$. Identifying $z^*_j\in\CC^{n**}$ with $e_j\in\CC^n$, we will think of ${\mathbf f}$ as the following element of $\CC[e_1,\dots,e_n]_{n(d-2)}$:
\begin{equation}
\sum_{i_1+\cdots+i_n=n(d-2)}\frac{(n(d-2))!}{i_1!\cdots i_n!}\mu_{i_1,\dots,i_n}(\bbf)
e_1^{i_1}\cdots e_n^{i_n}.\label{assocformexpp3}
\end{equation}
We call ${\mathbf f}$ given by expression (\ref{assocformexpp3}) the {\it associated form}\, of $\bbf$. Clearly, the associated form of $f\in X_n^d$ is the associated form of the gradient $(f_{z_{{}_1}},\dots,f_{z_{{}_n}})\in Y_n^{d-1}$.

We note that the associated form of $\bbf \in Y_n^{d-1}$ arises from the following invariantly defined map
$$
\fm/\fm^2 \to \Soc(M_\bbf),\quad x \mapsto y^{n(d-2)},\label{coordinatefree1}
$$
with $y\in\fm$ being any element that projects to $x\in\fm/\fm^2$. Indeed, ${\mathbf f}$ is derived from this map by identifying the target with $\CC$ via $\omega_{\bbf}$ and the source with $\CC^{n*}$ by mapping the image of $\widehat{z}_j$ in $\fm/\fm^2$ to the element $z_j$ of the basis $z_1,\dots,z_n$ of $\CC^{n*}$.

Again, it is not hard to show that each $\mu_{i_1,\dots,i_n}$ is a regular function on the affine variety $Y_n^{d-1}$ (cf.~\cite[the proof of Proposition 2.1]{I3}). Hence, we have 
$$
\mu_{i_1,\dots,i_n}=\frac{P_{i_1,\dots,i_n}}{\Res^{p_{i_1,\dots,i_n}}}\label{formulaformus1}
$$
for some $P_{i_1,\dots,i_n}\in\CC[(\CC[z_1,\dots,z_n]_{d-1})^{\oplus n}]$ and nonnegative integer $p_{i_1,\dots,i_n}$. Thus, we arrive at the morphism
$$
\Psi \co Y_n^{d-1}\ra \CC[e_1,\dots,e_n]_{n(d-2)},\quad \bbf\mapsto {\mathbf f}
$$
of affine algebraic varieties. Using the polar pairing, we may regard the associated form as an element of $\CC[z_1,\dots,z_n]_{n(d-2)}^*$, in which case $\Psi$ turns into a morphism from $Y_n^{d-1}$ to $\CC[z_1,\dots,z_n]_{n(d-2)}^*$; we call it $\tilde\Psi$.

The morphism $\tilde\Psi$ is easy to describe. For $\bbf\in Y_n^{d-1}$, denote by $\tilde\omega_\bbf$ the element of $\CC[z_1,\dots,z_n]_{n(d-2)}^*$ such that: 
\begin{itemize}

\item[(i)] $\ker\tilde\omega_\bbf=W_\bbf$ with $W_\bbf$ introduced in (\ref{subspace1}), and 

\vspace{0.1cm}
\item[(ii)]  $\tilde\omega_\bbf(\Jac(\bbf))=1$. 

\end{itemize}

\noindent Clearly, $\mu_{i_1,\dots,i_n}(\bbf)=\tilde\omega_\bbf(z_1^{i_1}\cdots z_n^{i_n})$ for $i_1+\dots+i_n=n(d-2)$. We have a fact analogous to Proposition \ref{newmap}:

\begin{proposition}\label{newmap1}
The morphism
$$
\tilde\Psi: Y_n^{d-1}\to \CC[z_1,\dots,z_n]_{n(d-2)}^*
$$
sends an $n$-tuple $\bbf$ to $(n(d-2))!\,\tilde\omega_\bbf$.
\end{proposition}

We will now state the equivariance property of the morphisms $\Psi$, $\tilde\Psi$. First, notice that for any $k$ the group $\GL_n \times \GL_n$ acts on the vector space $(\CC[z_1,\dots,z_n]_k)^{\oplus n}$ via
$$
((C_1, C_2)  \bbf) (z_1,\dots,z_n) := \bbf ((z_1,\dots,z_n)C_1^{-T})C_2^{-1}\label{doubleaction}
$$
for $\bbf\in (\CC[z_1,\dots,z_n]_k)^{\oplus n}$ and $C_1,C_2 \in \GL_n$. We then have (see \cite[Lemma 2.7]{AI2}):

\begin{proposition}\label{L:equiv2} For every $\bbf\in Y_n^{d-1}$ and $C_1,C_2 \in \GL_n$ the following holds:
\begin{equation}
\begin{array}{l}
\displaystyle\Psi((C_1,C_2) \bbf)=\det(C_1 C_2)\Bigl (C_1 \Psi(\bbf )\Bigr)\,\,\hbox{and}\\
\vspace{-0.3cm}\\
\hspace{4cm}\tilde\Psi((C_1,C_2) \bbf)=\det(C_1 C_2)\Bigl(C_1 \tilde\Psi(\bbf)\Bigr). 
\end{array}\label{E:equiv2}
\end{equation}
\end{proposition}

We conclude this subsection by observing that the morphisms $\Phi$, $\tilde\Phi$ can be factored as
\begin{equation}
\Phi=\Psi\circ\nabla|_{X_n^d},\quad\tilde\Phi=\tilde\Psi\circ\nabla|_{X_n^d},\label{decomposition}
\end{equation}
where $\nabla$ is the {\it gradient morphism}:
\begin{equation}
\nabla: \CC[z_1,\dots,z_n]_d\to (\CC[z_1,\dots,z_n]_{d-1})^{\oplus n},\quad f\mapsto (f_{z_{{}_1}},\dots,f_{z_{{}_n}}).\label{gradientmorph}
\end{equation}

\noindent Later on, this factorization will prove rather useful.

\subsection{Macaulay inverse systems and the image of $\Psi$}

We will now interpret the morphism $\Psi$ in different terms. Recall that the algebra $\CC[e_1,\dots,e_n]$ is a $\CC[z_1,\dots,z_n]$-module via differentiation:  
$$
\begin{array}{l}
\displaystyle(g \diamond F) (e_1, \ldots, e_n) := g\left(\frac{\partial}{\partial e_1}, \ldots, \frac{\partial}{\partial e_n}\right)F(e_1, \ldots, e_n),\\
\vspace{-0.3cm}\\
\hspace{7cm}g\in\CC[z_1,\dots,z_n],\, F\in\CC[e_1,\dots,e_n]. \label{pp}
\end{array}
$$
Restricting this module structure to $\CC[z_1,\dots,z_n]_k\times\CC[e_1,\dots,e_n]_k$, we obtain the perfect polar pairing described in (\ref{polarpairing}).  

For any $F\in\CC[e_1,\dots,e_n]_k$, we now introduce a homogeneous ideal, called the {\it annihilator}\, of $F$, as follows:
$$
F^{\perp} := \{g\in \CC[z_1,\dots,z_n]: g \diamond F = 0 \},
$$
which is clearly independent of scaling and thus is well-defined for $F$ in the projective space $\PP\,\CC[e_1,\dots,e_n]_k$ (from now on, we will sometimes think of forms as elements of the corresponding projective spaces, which will be clear from the context). It is well-known that the quotient $\CC[z_1,\dots,z_n]/F^{\perp}$ is a standard graded local Artinian Gorenstein algebra of socle degree $k$ and the following holds (cf.~\cite[Lemmas 2.12, 2.14]{IK}, \cite[Proposition 4]{Em}):

\begin{proposition} \label{prop-correspondence}
The correspondence $F \mapsto \CC[z_1,\dots,z_n]/F^{\perp}$ induces a bijection
$$
\PP\,\CC[e_1,\dots,e_n]_k \to
\left\{ 
 	\begin{array}{l} 
		\text{local Artinian Gorenstein algebras $\CC[z_1,\dots,z_n]/I$}\\ 
		\text{of socle degree $k$, where the ideal $I$ is homogeneous}\\
		 \end{array} \right\}.
$$
\end{proposition}

We also note that the isomorphism classes of local Artinian Gorenstein algebras $\CC[z_1,\dots,z_n]/I$ of socle degree $k$, where the ideal $I$ is homogeneous, are in bijective correspondence with the linear equivalence classes (i.e., $\GL_n$-orbits) of nonzero elements of $\CC[e_1,\dots,e_n]_k$ (see \cite[Proposition 17]{Em} and cf.~\cite[formula (5.7)]{I4}). This correspondence is induced by the map $F \mapsto \CC[z_1,\dots,z_n]/F^{\perp}$, $F\in\CC[e_1,\dots,e_n]_k$.

Any form $F\in\CC[e_1,\dots,e_n]_k$ such that $F^{\perp}=I$ is called {\it a {\rm (}homogeneous{\rm )} Macaulay inverse system of\, $\CC[z_1,\dots,z_n]/I$} and its image in $\PP\,\CC[e_1,\dots,e_n]_k$ is called {\it the {\rm (}homogeneous{\rm )} Macaulay inverse system of\, $\CC[z_1,\dots,z_n]/I$}.

We have (see \cite[Proposition 2.11]{AI2}):

\begin{prop} \label{P:inverse-system} \it
For any $\bbf \in Y_n^{d-1}$, 
the associated form $\Psi(\bbf)$ is a Macaulay inverse system of the algebra $M_\bbf$.
\end{prop}

\noindent By Proposition \ref{P:inverse-system}, the morphism $\Psi$ can be thought of as a map assigning to every element $\bbf \in Y_n^{d-1}$ a particular Macaulay inverse system of the algebra $M_\bbf$. Similarly, $\Phi$ assigns to every element $f \in X_n^d$ a particular Macaulay inverse system of $M_f$.

Let $U_n^{n(d-2)} \subset \CC[e_1,\dots,e_n]_{n(d-2)}$ be the locus of forms $F$ such that the subspace $F^{\perp} \cap\CC[z_1,\dots,z_n]_{d-1}$ is $n$-dimensional and has a basis with nonvanishing resultant. A description of $U_n^{n(d-2)}$ was given in \cite[Theorem 3.5]{I6}. It follows from this description (and is easy to see independently) that $U_n^{n(d-2)}$ is locally closed in $\CC[e_1,\dots,e_n]_{n(d-2)}$, hence is a quasi-affine variety. By Proposition \ref{P:inverse-system} we have $\im(\Psi)\subset U_n^{n(d-2)}$. In fact, one can show that $U_n^{n(d-2)}$ is exactly the image of $\Psi$:
%\vspace{-0.09cm}

\begin{prop} \label{P:image} \it $\im(\Psi)=U_n^{n(d-2)}$. 
\end{prop}
%\vspace{-0.3cm}

\begin{proof} If $F \in U_n^{n(d-2)}$, then for the ideal $I\subset\CC[z_1,\dots,z_n]$ generated by the subspace $F^{\perp} \cap\CC[z_1,\dots,z_n]_{d-1}$ we have $I \subset F^{\perp}$. Hence, one has the inclusion $I_{n(d-2)} \subset F^{\perp}_{n(d-2)}$ of the $n(d-2)$th graded components of these ideals. As both $I_{n(d-2)}$ and $F^{\perp}_{n(d-2)}$ have codimension 1 in $\CC[z_1,\dots,z_n]_{n(d-2)}$, it follows that $I_{n(d-2)}=F^{\perp}_{n(d-2)}$. By Proposition \ref{P:inverse-system}, for any basis $\bbf$ of $F^{\perp} \cap\CC[z_1,\dots,z_n]_{d-1}$ the associated form $\Psi(\bbf)$ is proportional to $F$, and therefore $F\in\im(\Psi)$.\end{proof}
%\vspace{-0.4cm}

\noindent In the next subsection we will state a projectivized variant of this proposition.

\subsection{Projectivizations of $\Phi$ and $\Psi$} 
The constructions of the morphisms $\Phi$ and $\Psi$ can be projectivized. Let $\PP X_n^d$ be the image of $X_n^d$ in the projective space 
$\PP\,\CC[z_1,\dots,z_n]_d$; it consists of all lines spanned by forms with nonzero discriminant. The discriminant on $\CC[z_1,\dots,z_n]_d$ descends to a section of a line bundle over $\PP\,\CC[z_1,\dots,z_n]_d$, and $\PP X_n^d$ is the affine open subset of $\PP\,\CC[z_1,\dots,z_n]_d$ where this section does not vanish (see Subsection \ref{reviewGIT} for details). The definition of the associated form of a form in $X_n^d$ (or, alternatively, equivariance property (\ref{equivarphitildephi})) yields that the morphism $\Phi$ descends to an $\SL_n$-equivariant morphism
$$
\PP\Phi \co \PP X_n^d \to \PP\,\CC[e_1,\dots,e_n]_{n(d-2)}.
$$
By Proposition \ref{P:inverse-system}, the morphism $\PP\Phi$ can be regarded as a map assigning to every line ${\mathcal L}\in\PP X_n^d$\, {\it the}\, Macaulay inverse system of the algebra $M_f$, where $f$ is any form that spans ${\mathcal L}$. Notice that by Example \ref{E:example1} this morphism is not injective. 

Next, let $Z_n^{d-1}$ be the image of $Y_n^{d-1}$ in the Grassmannian $\Gr(n,\CC[z_1,\dots,z_n]_{d-1})$ of $n$-dimensional subspaces of $\CC[z_1,\dots,z_n]_{d-1}$; it consists of all $n$-dimensional subspaces of $\CC[z_1,\dots,z_n]_{d-1}$ having a basis with nonzero resultant. The resultant $\Res$ on $(\CC[z_1,\dots,z_n]_{d-1})^{\oplus n}$ descends to a section of a line bundle over the Grassmannian $\Gr(n, \CC[z_1,\dots,z_n]_{d-1})$, and $Z_n^{d-1}$ is the affine open subset of $\Gr(n,\CC[z_1,\dots,z_n]_{d-1})$ where this section does not vanish (see Subsection \ref{reviewGIT}). Equivariance property (\ref{E:equiv2}) shows that the morphism $\Psi$ induces an $\SL_n$-equivariant morphism
$$
\PP\Psi \co Z_n^{d-1} \to \PP\,\CC[e_1,\dots,e_n]_{n(d-2)}.
$$ 
By Proposition \ref{P:inverse-system}, the morphism $\PP\Psi$ can be thought of as a map assigning to every subspace in $\Gr(n,\CC[z_1,\dots,z_n]_{d-1})$ the Macaulay inverse system of the algebra $M_\bbf$, with $\bbf=(f_1,\dots,f_n)$ being any basis of the subspace.

Proposition \ref{P:image} yields $\im(\PP\Psi)=\PP U_n^{n(d-2)}$, where $\PP U_n^{n(d-2)}$ is the image of $U_n^{n(d-2)}$ in the projective space $\PP\,\CC[e_1,\dots,e_n]_{n(d-2)}$. Clearly, $\PP U_n^{n(d-2)}$ is locally closed, hence is a quasi-projective variety. With a little extra effort one obtains (see \cite[Proposition 2.13]{AI2}):

\begin{prop} \label{P:imagehat} \it The morphism $\PP\Psi:Z_n^{d-1}\to\PP U_n^{n(d-2)}$ is an isomorphism.
\end{prop}

\begin{proof} The morphism $\chi \co \PP U_n^{n(d-2)} \to Z_n^{d-1}$ given by $F \mapsto F^{\perp} \cap \CC[z_1,\dots,z_n]_{d-1}$ yields the diagram
$$\xymatrix{
Z_n^{d-1} \ar[r]^{\hspace{-0.3cm}\PP\Psi} \ar[rd]^{{\mathrm {id}}}		&\PP U_n^{n(d-2)}\ar[d]^{\chi} \\
											& Z_n^{d-1},
}$$		
which is commutative by Proposition \ref{P:inverse-system}. As $\chi$ is separated, it follows that $\PP\Psi$ is an isomorphism (see \cite[Remark 9.11]{GW}).\end{proof}

\noindent By Proposition \ref{P:imagehat}, the map $\PP\Psi:Z_n^{d-1}\to\PP\CC[e_1,\dots,e_n]_{n(d-2)}$ is a locally closed immersion, i.e., an isomorphism onto a locally closed subset of $\PP\CC[e_1,\dots,e_n]_{n(d-2)}$.

Next, consider an open subset of $\CC[z_1,\dots,z_n]_d$:
$$
W_n^d:=\CC[z_1,\dots,z_n]_d\setminus\left\{f\in\CC[z_1,\dots,z_n]_d:f_{z_{{}_1}},\dots,f_{z_{{}_n}}\,\hbox{are linearly dependent}\right\}.
$$
Clearly, we have $X_n^d\subset W_n^d$. The gradient morphism $\nabla$ introduced in (\ref{gradientmorph}) indices the morphism
$$
W_n^d\to\Gr(n,\CC[z_1,\dots,z_n]_{d-1}),\,\,\, f\mapsto\langle f_{z_{{}_1}},\dots,f_{z_{{}_n}}\rangle,
$$
where $\langle{}\cdot{}\rangle$ denotes linear span. This morphism descends to an $\SL_n$-equivariant morphism
$$
\PP\nabla\co\PP W_n^d\to\Gr(n,\CC[z_1,\dots,z_n]_{d-1}),
$$
where $\PP W_n^d$ is the (open) image of $W_n^d$ in the projective space $\PP\,\CC[z_1,\dots,z_n]_d$. Clearly, $\PP\nabla$ maps $\PP X_n^d$ into $Z_n^{d-1}$, and from (\ref{decomposition}) we obtain
\begin{equation}
\PP\Phi=\PP\Psi\circ\PP\nabla|_{\PP X_n^d}.\label{decommpossition}
\end{equation}
This factorization will be of importance to us in relation with Conjecture \ref{conj2}.

\section{Results and open problems related to Conjecture \ref{conj2}}\label{results}
\setcounter{equation}{0}

\subsection{Review of Geometric Invariant Theory}\label{reviewGIT} We start this section by giving a brief overview of some of the concepts of Geometric Invariant Theory, or GIT. The principal reference for GIT is \cite{MFK}, but we will follow the more elementary expositions given in \cite{Ne} and \cite[Chapter 9]{LR}.

First of all, recall that an (affine) algebraic group $G$ is called {\it reductive}\, if its unipotent radical is trivial. Since we only consider algebraic groups over $\CC$, this condition is equivalent to $G$ being the affine algebraic complexification of a compact group $K$; in this case 
%$K$ is contained in a maximal compact subgroup $\hat K$ of $G$ of the same dimension and the inclusion 
$K\xhookrightarrow{}G$ is the universal complexification of $K$ (see, e.g., \cite[p.~247]{VO}, \cite[Theorems 5.1, 5.3]{Ho}). The groups $\GL_n$ and $\SL_n$ are examples of reductive groups being the affine algebraic complexifications of $\U_n$ and $\SU_n$, respectively. 

Let $X$ be an algebraic variety and $G$ a reductive group acting algebraically on $X$. For any open $G$-invariant subset $U$ of $X$ denote by ${\mathcal O}_X(U)^G$ the algebra of $G$-invariant regular functions of $X$ on $U$. A {\it good quotient} of $X$ by $G$ is a pair $(Z,\pi)$, where $Z$ is an algebraic variety and $\pi \co X\ra Z$ is a morphism such that:
\begin{enumerate}
\item[(P1)] $\pi$ is surjective;
\item[(P2)] $\pi$ is $G$-invariant, i.e., $\pi(gx)=\pi(x)$ for all $g\in G$ and $x\in X$;
\item[(P3)] $\pi$ is affine, i.e., the inverse image of an open affine subset of $Z$ is an open affine subset of $X$;
\item[(P4)]  the induced map
$$
\pi^*\co\mathcal{O}_Z(U) \to \mathcal{O}_X(\pi^{-1}(U))^G,\quad f\mapsto f\circ\pi
$$
is an isomorphism for every open subset $U\subset Z$.
\end{enumerate} 

The good quotient $Z$, if exists, possesses the following additional properties: 
\begin{enumerate}
\item[(P5)] for $x,x'\in X$ one has $\pi(x)=\pi(x')$ if and only if $\overline{G \cdot x}\cap\overline{G \cdot x'}\ne\emptyset$ (where $G \cdot x$ is the $G$-orbit of $x$), and every fiber of $\pi$ contains exactly one closed $G$-orbit (the unique orbit of minimal dimension); hence if $S_1$, $S_2$ are closed disjoint $G$-invariant subsets of $X$, then $\pi(S_1)\cap\pi(S_2)=\emptyset$;
\item[(P6)] if $U\subset X$ is a {\it saturated}\, open subset, (i.e., an open subset satisfying\linebreak $U=\pi^{-1}(\pi(U))$), then $\pi(U)$ is open and $(\pi(U),\pi|_U)$ is a good quotient of $U$;
\item[(P7)] if $A$ is a $G$-invariant closed subset of $X$, then $\pi(A)$ is closed in $Z$ and $(\pi(A),\pi|_A)$ is a good quotient of $A$;
\item[(P8)] if $X$ is normal, so is $Z$;
\item[(P9)] if $Y$ is an algebraic variety and $\varphi \co X\ra Y$ is a $G$-invariant morphism, then there exists a unique morphism $\tau_{\varphi} \co Z\ra Y$ such that $\varphi=\tau_{\varphi}\circ\pi$.
\end{enumerate}

In most situations, the construction of the morphism $\pi$ will be clear from the context, therefore we usually apply the term \lq\lq good quotient\rq\rq\, to the variety $Z$ rather than the pair $(Z,\pi)$. A good quotient, if exists, is unique up to isomorphism and is denoted by $X\gitq G$. If every fiber of $\pi$ consists of a single (closed) orbit, the quotient $X\gitq G$ is called {\it geometric}.

We will now describe two cases when good quotients are known to exist.
\vspace{0.25cm}

{\bf Case 1.} Assume that $X$ is an affine variety, so $X=\Spec\CC[X]$, where the coordinate ring $\CC[X]$ is finitely generated. Clearly, $G$ acts on $\CC[X]$, and this action is rational (see, e.g., \cite[p.~47]{Ne} for the definition of a rational action on an algebra). We now note that over $\CC$ the condition of reductivity for affine algebraic groups is equivalent to those of {\it linear reductivity}\, and {\it geometric reductivity}\, (see \cite[pp.~96--98]{LR} for details). Then by the Gordan-Hilbert-Mumford-Nagata theorem (see \cite{G}, \cite{Hi1}, \cite{Hi2}, \cite[p.~29]{MFK}, \cite{Na}), the algebra of invariants $\CC[X]^G$ is finitely generated. Choose generators $f_1,\dots,f_m$ of $\CC[X]^G$ and set
$
\pi:= (f_1,\dots,f_m): X\ra \CC^m.
$
Next, consider the ideal
\begin{equation}
I:=\{g\in\CC[w_1,\dots,w_m]: g\circ\pi=0\}.\label{idealforquotient}
\end{equation}
Clearly, $I$ is a radical ideal in $\CC[w_1,\dots,w_m]$ and $\CC[X]^G\simeq\CC[w_1,\dots,w_m]/I$. Let
\begin{equation}
Z:=\{w\in\CC^m: g(w) = 0\,\,\hbox{for all}\,\, g\in I\}.\label{affinequotient}
\end{equation}
It can be shown that the affine variety $Z$ is a good quotient of $X$. In other words, one has $X\gitq G=\Spec\CC[X]^G$.

If $V$ is a vector space over $\CC$, then $V\setminus\{0\}\gitq\CC^*$ is the projective space $\PP V$, with $\pi:V\setminus\{0\}\to\PP V$ being the natural projection. Note that every $\CC^*$-invariant open subset of $V\setminus\{0\}$ is saturated with respect to $\pi$. Hence, by property (P6) we see $\PP X_n^d=X_n^d\gitq\CC^*$ and $\PP W_n^d=W_n^d\gitq\CC^*$. Also, using properties (P6) and (P7) one observes $\PP U_n^{n(d-2)}=U_n^{n(d-2)}\gitq\CC^*$. 

Let $N:=\dim_{\CC}V$. For $\ell\le N$ setting
$$
S(\ell,V):=V^{\oplus \ell}\setminus\{(v_1,\dots,v_\ell)\in V^{\oplus \ell}: v_1,\dots, v_\ell\,\hbox{are linearly dependent}\}, 
$$
we have $\Gr(\ell,V)=S(\ell,V)\gitq\GL_\ell$ with
$$
\pi\co S(\ell,V)\to \Gr(\ell,V),\quad (v_1,\dots,v_\ell)\mapsto \langle v_1,\dots,v_\ell\rangle.
$$
Since $Y_n^{d-1}$ is saturated in $S(n,\CC[z_1,\dots,z_n]_{d-1})$, it follows that $Z_n^{d-1}=Y_n^{d-1}\gitq \GL_n$.
   
{\bf Case 2.} To describe this case, we need to give some definitions. Let $G$ be a reductive group with a linear representation $G\rightarrow\GL(V)$ on a vector space $V$ of dimension $N$, and $X\subset V$ a $G$-invariant affine algebraic subvariety with the algebraic action of $G$ induced from that on $V$. A point $x\in X$ is called {\it semistable}\, if the closure of the orbit $G\cdot x$ does not contain $0$, {\it polystable}\, if $x\ne 0$ and $G\cdot x$ is closed, and {\it stable}\, if $x$ is polystable and $\dim G\cdot x=\dim G$ (or, equivalently, the stabilizer of $x$ is zero-dimensional). The three loci are denoted by $X^{\ss}$, $X^{\ps}$, and $X^{\s}$, respectively. Clearly, $X^{\s}\subset X^{\ps} \subset X^{\ss}$.

Let now $X \subset\PP V$ be a $G$-invariant projective algebraic variety with the algebraic action of $G$ induced from that on $\PP V$. Then the semistability, polystability and stability of a point $x\in X$ are understood as the corresponding concepts for some (hence every) point $\hat x$ lying over $x$ in the affine cone $\hat X\subset V$ over $X$. We denote the three loci by $X^{\ss}$, $X^{\ps}$, and $X^{\s}$, respectively. One has $X^{\s}\subset X^{\ps}\subset X^{\ss}$. The loci $X^{\s}$ and $X^{\ss}$ are open subsets of $X$, and the following holds:
$$
\begin{array}{l}
X^{\ps}=\{x\in X^{\ss}: \hbox{$G\cdot x$ is closed in $X^{\ss}$}\},\\
\vspace{-0.3cm}\\
X^{\s}= \{x\in X^{\ss}: \hbox{$G\cdot x$ is closed in $X^{\ss}$ and $\dim G\cdot x=\dim G$}\}.
\end{array}
$$

Choose coordinates $x_1,\dots,x_N$ in $V$. Then by \cite[Proposition 9.5.2.2]{LR}, the semi\-stability of a point $x\in X$ is characterized by the existence of a $G$-invariant homogeneous form of positive degree in  $x_1,\dots,x_N$ nonvanishing at some (hence every) lift $\hat x$ of $x$. In fact, for any nonnegative integer $k$ any element of $\CC[x_1,\dots,x_N]_k$ can be identified with a global section of the $k$th tensor power $H^{\otimes k}$ of the {\it hyperplane line bundle}\, $H$ on $\PP V$ (see, e.g., \cite[Example 13.16]{GW}). Therefore the condition of the nonvanishing of a $G$-invariant homogeneous form at $\hat x$ is equivalent to that of the nonvanishing at $x$ of the corresponding global $G$-invariant section of a power\linebreak of $H$.  

For instance, let us think of the Grassmannian $\Gr(n,\CC[z_1,\dots,z_n]_{d-1})$ as the projective variety in $\PP\bigwedge^n\CC[z_1,\dots,z_n]_{d-1}$ obtained via the Pl\"ucker embedding. It then follows that $Z_n^{d-1}$ lies in $\Gr(n,\CC[z_1,\dots,z_n]_{d-1})^{\ss}$ since $Z_n^{d-1}$ consists exactly of the elements of $\Gr(n,\CC[z_1,\dots,z_n]_{d-1})$ at which the resultant $\Res$, understood as the restriction of a global $\SL_n$-invariant section of $H^{\otimes(d-1)^{n-1}}$ to $\Gr(n,\CC[z_1,\dots,z_n]_{d-1})$, does not vanish. This description of $Z_n^{d-1}$ is a consequence of \cite[p.~257, Corollary 2.3 and p.~427, Proposition 1.1]{GKZ}. Similarly, we have $\PP X_n^d\subset\PP\CC[z_1,\dots,z_n]_d^{\ss}$ since $\PP X_n^d$ consists exactly of the elements of $\PP\CC[z_1,\dots,z_n]_d$ at which the discriminant $\Delta$, understood as a global $\SL_n$-invariant section of $H^{\otimes n(d-1)^{n-1}}$, does not vanish. In fact, by \cite[Proposition 4.2]{MFK}, we have $\PP X_n^d\subset\PP\CC[z_1,\dots,z_n]_d^{\s}$. 

Returning to the general setting, let $L:=H|_X$. One can show that for all sufficiently high $k$ every global section of $L^{\otimes k}$ is the restriction of a global section of $H^{\otimes k}$ to $X$ (see \cite[p.~13]{Ne}). Consider the algebra
$$
R:=\bigoplus_{k=0}^{\infty}R_k,
$$
where $R_k:=\Gamma(X,L^{\otimes k})$. It is finitely generated (see \cite[Chapter III, Theorem 5.2]{Hart} and \cite[Proposition 7.45]{GW}), and we have $X=\Proj R$ (see \cite[Proposition 13.74]{GW}).  

Now, the group $G$ rationally acts on $R$, and by the Gordan-Hilbert-Mumford-Nagata theorem the algebra of global $G$-invariant sections
$$
R^G=\bigoplus_{k=0}^{\infty}R_k^G
$$
is finitely generated over $R_0^G=\CC$. By \cite[Chapter III, \S1.3, Proposition 3]{Bo} (see also \cite[Lemma 13.10 and Remark 13.11]{GW}), we can find $p$ such that the {\it Veronese subalgebra} 
$$
R^{G(p)}:=\bigoplus_{k=0}^{\infty}R_{kp}^G
$$
is generated in degree $1$ over $R_0^G$, namely $R^{G(p)}=R_0^G[R_p^G]$. Let $f_1,\dots,f_m$ be degree $1$ generators of $R^{G(p)}$, and consider the rational map
$$
\pi\co X\xdashrightarrow{}\PP^{m-1},\quad [x_1:\dots:x_N]\mapsto [f_1(x_1,\dots,x_N):\cdots:f_m(x_1,\dots,x_N)].
$$
By \cite[Proposition 9.5.2.2]{LR}, the indeterminacy locus of this rational map is exactly the complement to the semistable locus $X^{\ss}$, so $\pi$ is a morphism from $X^{\ss}$ to $\PP^{m-1}$. 

Now, consider the ideal $I$ defined by formula (\ref{idealforquotient}). This ideal is homogeneous and is generated by all forms $g$ in $w_1,\dots,w_m$ such that $g\circ\pi=0$. Clearly, $I$ is a radical ideal in $\CC[w_1,\dots,w_m]$ and $R^{G(p)}\simeq\CC[w_1,\dots,w_m]/I$. Then, analogously to (\ref{affinequotient}) we set
$$
Z:=\{[w_1:\cdots:w_m]\in\PP^{m-1}: g(w_1,\dots,w_m) = 0\,\,\hbox{for all}\,\, g\in I\}.
$$
It can be shown that the projective variety $Z$ is a good quotient of $X^{\ss}$. In other words, one has $X^{\ss}\gitq G=\Proj R^{G(p)}$ (cf.~\cite[Proposition 13.12]{GW}), and by \cite[Remark 13.7]{GW} we see $X^{\ss}\gitq G=\Proj R^G$.

As the open subset $X^{\s}\subset X^{\ss}$ is saturated, $\pi(X^{\s})\subset X^{\ss}\gitq G$ is a good quotient of $X^{\s}$; this quotient is quasi-projective and geometric. 

\subsection{Interpretation of Conjecture \ref{conj2} via GIT}

Recall that the image of the morphism $\PP\Psi$ coincides with $\PP U_n^{n(d-2)}$ (see Proposition \ref{P:imagehat}). By \cite[Theorem 1.2]{F} (see also \cite{FI}), the variety $\PP U_n^{n(d-2)}$ lies in $\PP\CC[e_1,\dots,e_n]_{n(d-2)}^{\ss}$. Properties (\ref{E:equiv2}) and (P9) then show that there exists a morphism
$$
\overline{\PP\Psi}\co Z_n^{d-1}\gitq\SL_n \to \PP\CC[e_1,\dots,e_n]_{n(d-2)}^{\ss}\gitq\SL_n
$$
of good GIT quotients by $\SL_n$ such that the following diagram commutes:
$$
\xymatrix{
Z_n^{d-1} \ar[r]^{\hspace{-0.8cm}{\PP\Psi}} \ar[d]_{\pi_2} 		& \PP\CC[e_1,\dots,e_n]_{n(d-2)}^{\ss} \ar[d]^{\pi_1} \\
Z_n^{d-1}\gitq\SL_n	\ar[r]^{\hspace{-1.2cm}{{\overline{\PP\Psi}}}}			&\PP\CC[e_1,\dots,e_n]_{n(d-2)}^{\ss}\gitq\SL_n
}
$$
(here and below we denote by $\pi_1,\pi_2,\dots$ the relevant quotient projections). Notice that $Z_n^{d-1}\gitq\SL_n$ is affine while $\PP\CC[e_1,\dots,e_n]_{n(d-2)}^{\ss}\gitq\SL_n$ is projective. 

Next, the morphism $\PP\nabla|_{\PP X_n^d}$ leads to a morphism of good affine GIT quotients
$$
\overline{\PP\nabla|_{\PP X_n^d}}\co\PP X_n^d \gitq\SL_n \to Z_n^{d-1}\gitq\SL_n
$$
and a commutative diagram
$$
\xymatrix{
\PP X_n^d \ar[r]^{\hspace{-0.4cm}{\PP\nabla|_{\PP X_n^d}}} \ar[d]_{\pi_3} 		& Z_n^{d-1} \ar[d]^{\pi_2} \\
\PP X_n^d\gitq\SL_n	\ar[r]^{\hspace{-0.4cm}{\overline{\PP\nabla|_{\PP X_n^d}}}}			&Z_n^{d-1}\gitq\SL_n.
}
$$

Recalling factorization (\ref{decommpossition}), we now see that $\PP\Phi$ maps $\PP X_n^d$ to the semistable locus $\PP\CC[e_1,\dots,e_n]_{n(d-2)}^{\ss}$ and that the morphism
$$
\overline{\PP\Phi}\co\PP X_n^d \gitq\SL_n\to\PP\CC[e_1,\dots,e_n]_{n(d-2)}^{\ss}\gitq\SL_n
$$
corresponding to the commutative diagram
$$
\xymatrix{
\PP X_n^d \ar[r]^{\hspace{-0.8cm}{\PP\Phi}} \ar[d]_{\pi_3} 		& \PP\CC[e_1,\dots,e_n]_{n(d-2)}^{\ss} \ar[d]^{\pi_1} \\
\PP X_n^d\gitq\SL_n	\ar[r]^{\hspace{-1.2cm}{{\overline{\PP\Phi}}}}			&\PP\CC[e_1,\dots,e_n]_{n(d-2)}^{\ss}\gitq\SL_n
}
$$
factors as
\begin{equation}
{\overline{\PP\Phi}}={\overline{\PP\Psi}}\circ\overline{\PP\nabla|_{\PP X_n^d}}.\label{decomposi1}
\end{equation}

We will now relate the above facts to Conjecture {\rm \ref{conj2}}. The following claim corrects the assertion made in \cite{AI2} that the positive answer to Question 3.1, stated therein, yields the conjecture. The claim that appears below has been suggested to us by M.~Fedorchuk.

\begin{claim}\label{claimconj}
\it In order to establish Conjecture {\rm \ref{conj2}} it suffices to show that ${\overline{\PP\Phi}}$ is an isomorphism onto a closed subset of an affine open subset of the GIT quotient\, $\PP\CC[e_1,\dots,e_n]_{n(d-2)}^{\ss}\gitq\SL_n$.
\end{claim}

\begin{proof} Let $U\subset\PP\CC[e_1,\dots,e_n]_{n(d-2)}^{\ss}\gitq\SL_n$ be an affine open subset and $A\subset U$ a closed subset such that
$$
{\overline{\PP\Phi}}\co \PP X_n^d\gitq\SL_n\to A
$$
is an isomorphism. Fix a $\GL_n$-invariant regular function $S$ on $X_n^d$. By property (P4) it is the pullback of a uniquely defined regular function $\bar S$ on $\PP X_n^d\gitq\SL_n$. Let $T$ be the push-forward of $\bar S$ to $A$ by means of ${\overline{\PP\Phi}}$. Since $A$ is closed in $U$ and $U$ is affine, the function $T$ extends to a regular function on $U$. The pull-back of this function by means of $\pi_1$ yields an $\SL_n$-invariant regular function on the dense open subset $\pi_1^{-1}(U)$ of $\PP\CC[e_1,\dots,e_n]_{n(d-2)}^{\ss}$, hence a $\GL_n$-invariant rational function $R$ on $\CC[e_1,\dots,e_n]_{n(d-2)}$. Clearly, $R$ is defined on $\im(\Phi)$ and $R\circ\Phi=S$ as required by Conjecture {\rm \ref{conj2}}.\end{proof}

Factorization (\ref{decomposi1}) yields that in order to show that the map ${\overline{\PP\Phi}}$ satisfies the condition stated in Claim \ref{claimconj}, it suffices to prove the following:
\begin{itemize}

\item[(C1)] $\overline{\PP\nabla|_{\PP X_n^d}}$ is a closed immersion, i.e., an isomorphism onto a closed subset of $Z_n^{d-1}\gitq\SL_n$;

\item[(C2)] ${\overline{\PP\Psi}}$ is an isomorphism onto a closed subset of an affine open subset of $\PP\CC[e_1,\dots,e_n]_{n(d-2)}^{\ss}\gitq\SL_n$.

\end{itemize}

Neither of conditions (C1), (C2) has been established in full generality, so we state:

\begin{openprob}\label{prob1}
\it Prove that conditions {\rm(C1)} and {\rm (C2)} are satisfied for all $n\ge 2$, $d\ge 3$.
\end{openprob}

\noindent Below, we will list known results leading towards settling these conditions.

\subsection{Results concerning conditions (C1) and (C2)} We start with condition (C1). First, we note that the locus $\PP W_n^d$, where the morphism $\PP\nabla$ is defined, contains the semistable locus $\PP\CC[z_1,\dots,z_n]_d^{\ss}$ (see \cite[p.~452]{F}). Next, it is shown in \cite[Theorem 1.7]{F} that the morphism $\PP\nabla$ preserves semistability, i.e., that the element\linebreak $\PP\nabla(f)\in\Gr(n,\CC[z_1,\dots,z_n]_{d-1})$ is semistable whenever $f\in\PP W_n^d$ is semistable. Denoting the restriction of $\PP\nabla$ to $\PP\CC[z_1,\dots,z_n]_d^{\ss}$ by the same symbol, we thus have a morphism of good GIT quotients
$$
\overline{\PP\nabla}\co\PP\CC[z_1,\dots,z_n]_d^{\ss} \gitq\SL_n \to \Gr(n,\CC[z_1,\dots,z_n]_{d-1})^{\ss}\gitq\SL_n
$$
and a commutative diagram
$$
\xymatrix{
\PP\CC[z_1,\dots,z_n]_d^{\ss} \ar[r]^{\hspace{-0.4cm}{\PP\nabla}} \ar[d]_{\pi_5} 		& \Gr(n,\CC[z_1,\dots,z_n]_{d-1})^{\ss} \ar[d]^{\pi_4} \\
\PP\CC[z_1,\dots,z_n]_d^{\ss} \gitq\SL_n	\ar[r]^{\hspace{-0.8cm}{\overline{\PP\nabla}}}			&\Gr(n,\CC[z_1,\dots,z_n]_{d-1})^{\ss}\gitq\SL_n.
}
$$

As each of the subsets $\PP X_n^d$ and $Z_n^{d-1}$ is defined as the loci where an $\SL_n$-invariant section of a power of the hyperplane bundle does not vanish, these subsets are saturated. Hence we can assume that the projection $\pi_3$ is the restriction of $\pi_5$ to $\PP X_n^d$, the projection $\pi_2$ is the restriction of $\pi_4$ to $Z_n^{d-1}$,  and $\overline{\PP\nabla|_{\PP X_n^d}}$ is the restriction of $\overline{\PP\nabla}$ to $\PP X_n^d\gitq\SL_n\subset \PP\CC[z_1,\dots,z_n]_d^{\ss} \gitq\SL_n $.

We now state: 

\begin{theorem}\label{nablefedorchuk}\cite[Proposition 2.1, part (3)]{F}
The morphism $\overline{\PP\nabla|_{\PP X_n^d}}$ is finite and injective. 
\end{theorem}

By Theorem \ref{nablefedorchuk} and Zariski's Main Theorem (see \cite[Corollary 17.4.8]{TY}), condition (C1) will follow if we establish the normality of the (closed) image of $\overline{\PP\nabla|_{\PP X_n^d}}$ in $Z_n^{d-1}\gitq\SL_n$. Thus, condition (C1) is a consequence of a positive answer to:

\begin{openprob}\label{probnormality}
\it Show that the image of $\overline{\PP\nabla|_{\PP X_n^d}}$ is normal for all $n\ge 2$, $d\ge 3$.
\end{openprob}

While the above problem remains open in full generality, for the case $n=2$ we have the following result, which even gives the normality of $\im(\overline{\PP\nabla})$: 

\begin{theorem}\label{isaevalper}\cite[Corollaries 5.5 and 6.6]{AI2}
Assume that $n=2$. Then the morphism $\overline{\PP\nabla}$ is finite and injective, and its image in $\Gr(2,\CC[z_1,z_2]_{d-1})^{\ss}\gitq\SL_2$ is normal.
\end{theorem}

Another positive result on condition (C1) concerns the case of ternary cubics (here $n=d=3$):  

\begin{proposition}\label{ternarycubicsnabla}
The image $\im(\overline{\PP\nabla|_{\PP X_3^3}})$ is a nonsingular curve in $Z_3^2\gitq\SL_3$.
\end{proposition}

\noindent Proposition \ref{ternarycubicsnabla} has never appeared in print as stated but easily follows from other published facts. Details will be given in Section \ref{S:binaryquarticternarycubics} (see Remark \ref{Rem:binaryquarticternarycubics}).

Next, we will discuss condition (C2). First of all, the following holds:

\begin{theorem}\label{psilocclosedimmer}\cite[Corollary 5.4]{FI}
The morphism $\overline{\PP\Psi}$ is a locally closed immersion.
\end{theorem}

\begin{proof} The proof is primarily based on \cite[Theorem 1.2]{FI}, the main result of \cite{FI}, which states that $\PP\Psi$ maps polystable points to polystable points. Once this difficult fact has been established, we proceed as follows.

Recall that by Proposition \ref{P:imagehat} the morphism $\PP\Psi$ is a locally closed immersion, specifically, is an isomorphism onto the $\SL_n$-invariant locally closed subset $\PP U_n^{n(d-2)}$ of the semistable locus $\PP\CC[e_1,\dots,e_n]_{n(d-2)}^{\ss}$. Therefore, property (P5) implies that $\overline{\PP\Psi}$ is injective. 

Next, consider the closure $Z$ of $\PP U_n^{n(d-2)}$ in $\PP\CC[e_1,\dots,e_n]_{n(d-2)}^{\ss}$. Clearly, $Z$ is\linebreak $\SL_n$-invariant. Then by property (P7) we see that $\pi_1(Z)$ is closed in the projective variety $\PP\CC[e_1,\dots,e_n]_{n(d-2)}^{\ss}\gitq\SL_n$ and is a good quotient of $Z$. Since $\PP U_n^{n(d-2)}$ is locally closed in $\PP\CC[e_1,\dots,e_n]_{n(d-2)}^{\ss}$, it is open in $Z$. Let us show that $\PP U_n^{n(d-2)}$ is saturated in $Z$ as well. Fix $a\in\PP U_n^{n(d-2)}$ and let $O$ be the unique closed\linebreak $\SL_n$-orbit in $Z$ that lies in the closure of $\SL_n\cdot\, a$ in $Z$ (see  property (P5)). Set $\tilde a:=(\PP\Psi)^{-1}(a)\in Z_n^{d-1}$ and consider the closed $\SL_n$-orbit $\tilde O$ that lies in the closure of $\SL_n\cdot\,\tilde a$ in $Z_n^{d-1}$. Appealing to \cite[Theorem 1.2]{FI} once again, we see that $\PP\Psi(\tilde O)$ is a closed $\SL_n$-orbit in $Z$ and that $\PP\Psi(\tilde O)$ lies in the closure of $\SL_n\cdot\,a$ in $Z$. It then follows that $O=\PP\Psi(\tilde O)$, so $O$ is contained in $\PP U_n^{n(d-2)}$. Since $\PP U_n^{n(d-2)}$ is open in $Z$, the $\SL_n$-orbit of every point of $Z$ that contains $O$ in its closure in fact lies in $\PP U_n^{n(d-2)}$. This shows that $\PP U_n^{n(d-2)}$ is saturated in $Z$ as claimed.

By property (P6) we then have that $\pi_1\left(\PP U_n^{n(d-2)}\right)$ is open in $\pi_1(Z)$ and is a good quotient of $\PP U_n^{n(d-2)}$. Now, recall that $\PP U_n^{n(d-2)}$ is isomorphic to the smooth variety $Z_n^{d-1}$, hence is normal. By property (P8) we therefore see that $\pi_1\left(\PP U_n^{n(d-2)}\right)$ is a normal variety. Zariski's Main Theorem now implies that $\overline{\PP\Psi}$ is an isomorphism onto $\im(\overline{\PP\Psi})=\pi_1\left(\PP U_n^{n(d-2)}\right)$, hence is a locally closed immersion.\end{proof}

\noindent Despite the fact that $\PP\Phi$ is not injective, factorization (\ref{decomposi1}) and Theorems \ref{nablefedorchuk}, \ref{psilocclosedimmer} imply

\begin{corollary}\label{barphiinj}
The morphism $\overline{\PP\Phi}$ is injective.
\end{corollary}

Note that Theorem \ref{psilocclosedimmer} states that the map $\overline{\PP\Psi}$ is an isomorphism onto the closed subset $\pi_1\left(\PP U_n^{n(d-2)}\right)$ of an open subset of $\PP\CC[e_1,\dots,e_n]_{d(n-2)}^{\ss}\gitq\SL_n$ but does not assert that the open subset may be chosen to be affine as required by condition (C2). We will now make the following observation:

\begin{proposition}\label{invargeneral}
Suppose that there exists a homogeneous $\SL_n$-invariant $J$ on the space $\CC[e_1,\dots,e_n]_{n(d-2)}$ such that $U_n^{n(d-2)}$ is a closed subset of the complement to the zero locus of $J$. Then $\pi_1\left(\PP U_n^{n(d-2)}\right)$ is a closed subset of an affine open subset of $\PP\CC[e_1,\dots,e_n]_{d(n-2)}^{\ss}\gitq\SL_n$, hence condition {\rm (C2)} is satisfied.
\end{proposition}

\begin{proof} Let $U$ be the affine open subset of $\PP\CC[e_1,\dots,e_n]_{n(d-2)}$ that consists of all elements of $\PP\CC[e_1,\dots,e_n]_{n(d-2)}$ at which $J$, understood as a global $\SL_n$-invariant section of a power of the hyperplane bundle, does not vanish. Then $\PP U_n^{n(d-2)}$ is a closed subset of $U$. Since $U$ is a saturated open subset of  $\PP\CC[e_1,\dots,e_n]_{n(d-2)}^{\ss}$, by property (P6) it follows that $\pi_1(U)$ is open in $\PP\CC[e_1,\dots,e_n]_{n(d-2)}^{\ss}\gitq\SL_n$ and is a good quotient of $U$. As $U$ is affine, its good quotient is also affine, hence $\pi_1(U)$ is an affine open subset of $\PP\CC[e_1,\dots,e_n]_{n(d-2)}^{\ss}\gitq\SL_n$. Since $\PP U_n^{n(d-2)}$ is a closed $\SL_n$-invariant subset of of $U$, by property (P7) we see that $\pi_1\left(\PP U_n^{n(d-2)}\right)$ is a closed subset of $\pi_1(U)$.\end{proof} 

Following Proposition \ref{invargeneral}, we now state:

\begin{openprob}\label{openprobexistinvari}
\it Show that for all $n\ge 2$, $d\ge 3$ one can find a homogeneous $\SL_n$-invariant on the space $\CC[e_1,\dots,e_n]_{n(d-2)}$ such that $U_n^{n(d-2)}$ is a closed subset of the complement to its zero locus. 
\end{openprob}

\noindent We note that in \cite[Theorem 3.5]{I6} we constructed a hypersurface in an affine variety $A$ containing $U_n^{n(d-2)}$ that does not intersect $U_n^{n(d-2)}$ but it is not clear whether this hypersurface comes from the zero set of an $\SL_n$-invariant.

While Problem \ref{openprobexistinvari} remains open in full generality, it has been solved in the cases $n=2$ and $n=d=3$. To discuss the case $n=2$, let us recall that the {\it catalecticant}\, of a binary form 
$$
f=\sum_{i=0}^{2N}{2N \choose i}a_iw_1^{2N-i}w_2^i
$$
of even degree $2N$ is
\begin{equation}
\Cat(f):=\det\left(\begin{array}{cccc}
a_0 & a_1 & \dots & a_N\\
a_1 & a_2 & \dots & a_{N+1}\\
\vdots & \vdots & \ddots & \vdots\\
a_N & a_{N+1} & \dots & a_{2N}
\end{array}
\right).\label{catalectform}
\end{equation}
It is an $\SL_2$-invariant and does not vanish if and only if the $N+1$ partial derivatives of $f$ of order $N$ are linearly independent in $\CC[w_1,w_2]_N$ (see, e.g., \cite[Lemma 6.2]{K}). Alternatively, the set where the catalecticant is nonzero is the complement to the closure of the locus of forms in $\CC[w_1,w_2]_{2N}$ expressible as the sum of the $2N$th powers of $N$ linear forms (see, e.g., \cite[\S 208]{Ell} or  \cite[\S 187]{GY}). 

Notice that the catalecticant is defined on the target space of the morphism $\Psi \co Y_n^{d-1} \to \CC[e_1,e_2]_{2(d-2)}$. Let us denote the affine open subset of $\CC[e_1,e_2]_{2(d-2)}$ where $\Cat$ does not vanish by $V_2^{2(d-2)}$ and its image in $\PP\CC[e_1,e_2]_{2(d-2)}$ by $\PP V_2^{2(d-2)}$. Clearly, $\PP V_2^{2(d-2)}$ is the affine open subset of $\PP\CC[e_1,e_2]_{2(d-2)}$ that consists of all elements of $\PP\CC[e_1,e_2]_{2(d-2)}$ at which the catalecticant $\Cat$, understood as a global $\SL_2$-invariant section of $H^{\otimes(d-1)}$, does not vanish.

For binary forms the following holds:

\begin{theorem}\label{catnonzero}\cite[Proposition 4.3]{AI2}
One has $U_2^{2(d-2)}=V_2^{2(d-2)}$.
\end{theorem}

Next, we let $n=d=3$. Notice that in this case $n(d-2)=d=3$. Let $A_4$ be the {\it degree four Aronhold invariant}\, of ternary cubics. An explicit formula for $A_4$ can be found, for example, in \cite[p.~191]{Sal}. Namely, for a ternary cubic
$$
\begin{array}{l}
f(w_1,w_2,w_3)=aw_1^3+bw_2^3+cw_3^3+3dw_1^2w_2+3pw_1^2w_3+3qw_1w_2^2+\\
\vspace{-0.1cm}\\
\hspace{6cm}3rw_2^2w_3+3sw_1w_3^2+3tw_2w_3^2+6uw_1w_2w_3
\end{array}
$$
one has
\begin{equation}
\begin{array}{l}
A_4(f)=abcu-bcdp-acqr-abst-u(art+bps+cdq)+\\
\vspace{-0.3cm}\\
\hspace{1.5cm}aqt^2+ar^2s+bds^2+bp^2t+cd^2r+cpq^2-u^4+\\
\vspace{-0.3cm}\\
\hspace{1.5cm}2u^2(qs+dt+pr)-3u(drs+pqt)-q^2s^2-d^2t^2-\\
\vspace{-0.3cm}\\
\hspace{1.5cm}p^2r^2+dprt+prqs+dqst.
\end{array}\label{aronhold4}
\end{equation}
Let us denote the affine open subset of $\CC[e_1,e_2,e_3]_3$ where $A_4$ does not vanish by $V_3^3$ and its image in $\PP\CC[e_1,e_2,e_3]_3$ by $\PP V_3^3$. Clearly, $\PP V_3^3$ is the affine open subset of $\PP\CC[e_1,e_2,e_3]_3$ that consists of all elements of $\PP\CC[e_1,e_2,e_3]_3$ at which $A_4$, understood as a global $\SL_3$-invariant section of $H^{\otimes 4}$, does not vanish. 

For $n=d=3$ we have:

\begin{theorem}\label{ternarycubics}\cite[Proposition 4.1]{I5}
One has $U_3^3=V_3^3$.
\end{theorem}

Now, Claim \ref{claimconj}, Theorems \ref{isaevalper}, \ref{catnonzero}, \ref{ternarycubics} and Proposition \ref{ternarycubicsnabla} imply

\begin{corollary}\label{positivecor}
Conjecture {\rm \ref{conj2}} is valid for $n=2$ and for $n=d=3$.
\end{corollary}

In fact, in Section \ref{S:binaryquarticternarycubics} we will see that for $n=d=3$ factorizations (\ref{decomposition}), (\ref{decommpossition}), (\ref{decomposi1}) are not required. In this case, Conjecture \ref{conj2} can be obtained by studying the morphism $\Phi$ directly.

To conclude this subsection, we reiterate that in order to establish Conjecture \ref{conj2} in full generality it suffices to solve Open Problem \ref{prob1}, which in turn will follow from positive solutions to Open Problems \ref{probnormality} and \ref{openprobexistinvari}.

\subsection{A weak variant of Conjecture {\rm \ref{conj2}}} The initial version of Conjecture {\rm \ref{conj2}}, stated in \cite{EI}, did not contain the requirement that the $\GL_n$-invariant rational function $R$ be defined at every point of the image of $\Phi$. It was formulated as follows: 

\begin{conjecture}\label{conj3} For every regular $\GL_n$-invariant function $S$ on $X_n^d$ there exists a rational $\GL_n$-invariant function $R$ on $\CC[e_1,\dots,e_n]_{n(d-2)}$ such that $R\circ\Phi$ extends to a regular function on $X_n^d$ that coincides with $S$.
\end{conjecture}

Note, for instance, that for the morphism
$$
\varphi:\CC\to\CC^2,\quad z\mapsto (z,z)
$$
the rational function $R:=z_1/z_2$ is not defined at $(0,0)=\varphi(0)$ but the pullback $R\circ\varphi$ extends to the regular function $1$ on $\CC$. Conjecture \ref{conj3} does not rule out such situations, whereas Conjecture \ref{conj2} does. We stress that it is the stronger conjecture that is required for solving the reconstruction problem stated in the introduction. 

The weaker conjecture has turned out to be easier to settle:

\begin{theorem}\label{weakerconjsettled}\cite[Theorem 4.1]{AI1}
Conjecture {\rm\ref{conj3}} holds for all $n\ge 2$, $d\ge 3$.
\end{theorem}

\begin{proof} The case $n=2$, $d=3$ is trivial, and we exclude it from consideration (note also that in this situation the result is contained in Corollary \ref{positivecor}). Then we have $n(d-2)\ge 3$. The proof is based on the fact that $\im(\PP\Phi)$ intersects the locus of stable points $\PP\CC[e_1,\dots,e_n]_{n(d-2)}^{\s}$. In fact, in \cite[Proposition 4.3]{AI1} we show that $\im(\PP\Phi)$ contains an element with nonvanishing discriminant. Specifically, one can prove that the associated form of
\begin{equation}
f_0(z_1,\dots,z_n):=\left\{
\begin{array}{ll}
\displaystyle \sum_{1\le i<j<k\le n}z_iz_jz_k & \hbox{if $d=3$,}\\
\vspace{-0.1cm}\\
\displaystyle \sum_{1\le i<j\le n}(z_i^{d-2}z_j^2+z_i^2z_j^{d-2}) & \hbox{if $d\ge 4$.}
\end{array}
\right.\label{deff0}
\end{equation}
is nondegenerate. Once this nontrivial statement has been established, we proceed as follows.

Consider the nonempty open $\SL_n$-invariant subset
$$
U:=(\PP\Phi)^{-1}\left(\PP\CC[e_1,\dots,e_n]_{n(d-2)}^{\s}\right)\subset\PP X_n^d.
$$
Since $\PP X_n^d\subset\PP\CC[z_1,\dots,z_n]_d^{\s}$, the set $U$ is saturated in $\PP X_n^d$, hence $\pi_3(U)$ is a good geometric quotient of $U$, and we have the commutative diagram
$$\xymatrix{
U \ar[r]^{\hspace{-1cm}\PP\Phi|_U} \ar[d]^{\pi_{{}_3}|_U} & \PP\CC[e_1,\dots,e_n]_{n(d-2)}^{\s} \ar[d]^{\pi_{{}_1}|_{\PP\CC[e_1,\dots,e_n]_{n(d-2)}^{\s}}}\\
\pi_3(U)\ar[r]^{\hspace{0cm}\varphi} & Z,
}$$
where $Z:=\pi_1\left(\PP\CC[e_1,\dots,e_n]_{n(d-2)}^{\s}\right)$ is a good geometric quotient of the stable locus $\PP\CC[e_1,\dots,e_n]_{n(d-2)}^{\s}$ and $\varphi:=\overline{\PP\Phi}|_{\pi_3(U)}$. Recall that by Corollary \ref{barphiinj} the morphism $\varphi$ is injective.

Next, since the set $\varphi(\pi_3(U))$ is constructible, it contains a subset $W$ that is open in the closed irreducible subvariety ${\mathcal R}:=\overline{\varphi(\pi_3(U))}$ of $Z$. Let ${\mathcal R}_{\rm sing}$ be the singular locus of ${\mathcal R}$. Then $W\setminus {\mathcal R}_{\rm sing}$ is nonempty and open in ${\mathcal R}$ as well, and we choose an open subset $O\subset Z$ such that $W\setminus {\mathcal R}_{\rm sing}=O\cap {\mathcal R}$. Clearly, $W\setminus {\mathcal R}_{\rm sing}$ is closed in $O$. Next, choose $V\subset O$ to be an affine open subset intersecting $W\setminus {\mathcal R}_{\rm sing}$. Then the set $\tilde{\mathcal R}:=V\cap(W\setminus {\mathcal R}_{\rm sing})=V\cap {\mathcal R}$ is closed in $V$. Let $\tilde U:=\varphi^{-1}(V)=\varphi^{-1}(\tilde{\mathcal R})$. By construction
$$
\tilde\varphi:=\varphi|_{\tilde U} \co \tilde U \to \tilde{\mathcal R}\subset V
$$
is a bijective morphism from the open subset $\tilde U$ of $U$ onto the smooth variety $\tilde{\mathcal R}$. It now follows from Zariski's Main Theorem that $\tilde\varphi$ is an isomorphism.

We will now argue as in the proof of Claim \ref{claimconj}. Fix a $\GL_n$-invariant regular function $S$ on $X_n^d$. By property (P4), it is the pullback of a uniquely defined regular function $\bar S$ on $\PP X_n^d\gitq\SL_n$. Let $T$ be the push-forward of $\bar S|_{\tilde U}$ to $ \tilde{\mathcal R}$ by means of $\tilde\varphi$. Since $ \tilde{\mathcal R}$ is closed in $V$ and $V$ is affine, the function $T$ extends to a regular function on $V$. The pull-back of this function by means of $\pi_1|_{\PP\CC[e_1,\dots,e_n]_{n(d-2)}^{\s}}$ yields an $\SL_n$-invariant regular function on the dense open subset $\pi_1^{-1}(V)$ of $\PP\CC[e_1,\dots,e_n]_{n(d-2)}^{\s}$, hence a $\GL_n$-invariant rational function $R$ on $\CC[e_1,\dots,e_n]_{n(d-2)}$. Clearly, the composition $R\circ\Phi$ extends to a regular function on $\CC[z_1,\dots,z_n]_d$, and the extension coincides with $S$.\end{proof}

As we have seen, the main part of the proof of Theorem \ref{weakerconjsettled} is the existence of $f_0\in X_n^d$ such that $\Delta\Phi(f_0)\ne 0$ (see (\ref{weakerconjsettled})). The existence of such a form also insures that one can consider the iteration $\Phi^2$, viewed as a rational map from $\CC[z_1,\dots,z_n]_d$ to $\CC[z_1,\dots,z_n]_{n(n(d-2)-2)}$. This observation leads to the following natural question:

\begin{openprob}\label{probiterations}
\it Is the iteration $\Phi^k$ a well-defined rational map for all $k\in\NN${\rm ?}
\end{openprob}

\noindent In the next section we will look at the iterations of the projectivized map $\PP\Phi$ in two special cases: (i) $n=2$, $d=4$ and (ii) $n=d=3$. 

\section{The morphism $\PP\Phi$ for binary quartics and ternary cubics}\label{S:binaryquarticternarycubics}
\setcounter{equation}{0}

To further clarify the nature of the morphisms $\Phi$, $\PP\Phi$ and $\overline{\PP\Phi}$, in this section we will consider two special cases for which we will present results of explicit calculations. Notice that for all pairs $n,d$ (excluding the trivial case $n=2$, $d=3$) one has $n(d-2)\ge d$, and the equality holds exactly for the pairs $n=2$, $d=4$ and $n=3$, $d=3$. These are the situations we will focus on below. In particular, we will provide an independent verification of Conjecture \ref{conj2} in each of the two cases. We will also see that in these situations the morphism $\PP\Phi$ induces a unique equivariant involution on the variety $\PP X_n^d$ with one orbit removed and that the involution can be understood via projective duality. For convenience, everywhere in this section we will identify the algebras $\CC[z_1,\dots,z_n]_d$ and $\CC[e_1,\dots,e_n]_{2(d-2)}$ by means of identifying $z_j$ and $e_j$, thus the morphism $\PP\Phi$ will be regarded as a map from $\PP X_n^d$ to $\PP\CC[z_1,\dots,z_n]_{n(d-2)}^{\ss}$. In this interpretation, it has the following equivariance property:
\begin{equation}
\PP\Phi(C f)=C^{-T}\PP\Phi(f),\,\, f\in\PP X_n^d,\,\, C\in\SL_n\label{equivartype2}
\end{equation}
(see (\ref{equivarphitildephi})). The material that follows can be found in articles \cite{Ea}, \cite{EI}, \cite{I1}, \cite{I2}, \cite{AIK}.
 
\subsection{Binary quartics} \label{S:binary-quartics}
Let $n=2$, $d=4$. It is a classical result that every nondegenerate binary quartic is linearly equivalent to a quartic of the form
\begin{equation}
q_t(z_1,z_2):=z_1^4+tz_1^2z_2^2+z_2^4,\quad t\ne\pm 2\label{qt}
\end{equation}
(see \cite[\S 211]{Ell}). A straightforward calculation yields that the associated form of\linebreak $q_t$ is
\begin{equation}
{\mathbf q}_t(z_1,z_2):=\frac{1}{72(t^2-4)}(tz_1^4-12z_1^2z_2^2+tz_2^4).\label{bfqt}
\end{equation}
For $t\ne 0,\pm 6$ the quartic ${\mathbf q}_t$ is nondegenerate, and in this case the associated form of ${\mathbf q}_t$ is proportional to $q_t$, hence $(\PP\Phi)^2(q_t) = q_t$. As explained below, the exceptional quartics $q_0$, $q_6$, $q_{-6}$, are pairwise linearly equivalent.

It is easy to show that $\PP\CC[z_1,z_2]_4^{\ss}$ is the union of $\PP X_2^4$ (which coincides with $\PP\CC[z_1,z_2]_4^{\s}$) and two orbits that consist of strictly semistable elements:\linebreak $O_1:=\SL_2\cdot\, z_1^2z_2^2$, $O_2:=\SL_2\cdot \,z_1^2(z_1^2+z_2^2)$, of dimensions 2 and 3, respectively. Notice that $O_1$ is closed in $\PP\CC[z_1,z_2]_4^{\ss}$ and is contained in the closure of $O_2$. We then observe that $\PP\Phi$ maps $\PP X_2^4$ onto $\PP\CC[z_1,z_2]_4^{\ss}\setminus (O_2\cup O_3)$, where $O_3:=\SL_2\cdot\, q_0$ (as we will see shortly, $O_3$  contains the other exceptional quartics $q_6$, $q_{-6}$ as well). Also, notice that $\PP\Phi$ maps the 3-dimensional orbit $O_3$ onto the 2-dimensional orbit $O_1$. In particular, $\PP\Phi$ restricts to an equivariant involutive automorphism of $\PP X_2^4\setminus O_3$, which for $t\ne 0,\pm 6$ establishes a duality between the quartics $C q_t$ and $C^{-T} q_{-12/t}$ with $C\in\SL_2$, hence between the orbits $\SL_2\cdot\, q_t$ and $\SL_2\cdot\, q_{-12/t}$.  

In order to understand the induced map $\overline{\PP\Phi}$ of good GIT quotients, we note that the algebra of $\SL_2$-invariants $\CC[\CC[z_1,z_2]_4]^{\SL_2}$ is generated by the pair of elements $I_2$ and $\Cat$, where $I_2$ has degree 2 (see, e.g., \cite[\S\S 29, 30, 80]{Ell}). We have 
\begin{equation}
\Delta=I_2^3-27\,\Cat^2\label{deltabinquar}
\end{equation}
(see \cite[\S 81]{Ell}), and for a binary quartic of the form
$$
f(z_1,z_2)=az_1^4+6bz_1^2z_2^2+cz_2^4
$$
the value of $I_2$ is computed as 
\begin{equation}
\begin{array}{l}
I_2(f)=ac+3b^2.\label{form1}
\end{array}
\end{equation}
It then follows that the algebra $\CC[\PP X_2^4]^{\SL_2}\simeq\CC[X_2^4]^{\GL_2}$ is generated by
\begin{equation}
J:=\frac{I_2^3}{\Delta}.\label{form2}
\end{equation}
Therefore, $\PP X_2^4\gitq\SL_2$ is the affine space $\CC$, and $\PP\CC[z_1,z_2]_4^{\ss}\gitq\SL_2$ can be identified with $\PP^1$, where both $O_1$ and $O_2$ project to the point at infinity in $\PP^1$. 

Next, from formulas (\ref{catalectform}), \eqref{qt}, (\ref{deltabinquar}), \eqref{form1}, \eqref{form2} we calculate 
\begin{equation}
J(q_t)=\frac{(t^2+12)^3}{108(t^2-4)^2}\quad\hbox{for all $t\ne \pm 2$.}\label{form3}
\end{equation}
Clearly, \eqref{form3} yields
\begin{equation}
J(q_0)=J(q_6)=J(q_{-6})=1,\label{Jeq1}
\end{equation}
which implies that $q_0$, $q_6$, $q_{-6}$ are indeed pairwise linearly equivalent as claimed above and that the orbit $O_3$ is described by the condition $J=1$. 

Using (\ref{bfqt}), \eqref{form3} one obtains
$$
J({\mathbf q}_t)=\frac{J(q_t)}{J(q_t)-1}\quad\hbox{for all $t\ne 0,\pm 6$.}\label{jtransfbinquar}
$$
This shows that the map $\overline{\PP\Phi}$ extends to the automorphism $\varphi$ of $\PP^1$ given by
$$
\zeta\mapsto\frac{\zeta}{\zeta-1}.
$$
Clearly, one has $\varphi^{\,2}=\hbox{id}$, that is, $\varphi$ is an involution. It preserves $\PP^1\setminus\{1,\infty\}$, which corresponds to the duality between the orbits $\SL_2\cdot\, q_t$ and $\SL_2\cdot\, q_{-12/t}$ for $t\ne 0,\pm 6$ noted above. Further, $\varphi(1)=\infty$, which agrees with \eqref{Jeq1} and the fact that $O_3$ is mapped onto $O_1$. We also have $\varphi(\infty)=1$, but this identity has no interpretation at the level of orbits.  Indeed, $\PP\Phi$ cannot be equivariantly extended to an involution of $\PP\CC[z_1,z_2]_4^{\ss}$ as the fiber of the quotient $\PP\CC[z_1,z_2]_4^{\ss}\gitq\SL_2$ over $\infty$ contains $O_1$, which cannot be mapped onto $O_3$ since $\dim O_1<\dim O_3$.

Finally, an explicit calculation shows that $\Cat({\mathbf q}_t)\ne 0$ for all $t\ne\pm 2$ (cf.~Theorem \ref{catnonzero}). Consider the absolute invariant of binary quartics
$$
K:=\frac{I_2^3}{27\Cat^2}.
$$
It is then easy to see that $K({\mathbf q}_t)=J(q_t)$ for all $t\ne\pm 2$, which independently establishes Conjecture \ref{conj2} for $n=2$, $d=4$ (cf.~Corollary \ref{positivecor}).

\subsection{Ternary cubics} \label{S:cubics}
Let $n=d=3$. Every nondegenerate  ternary cubic is linearly equivalent to a cubic of the form
\begin{equation}
c_t(z_1,z_2,z_3):=z_1^3+z_2^3+z_3^3+tz_1z_2z_3,\quad t^3\ne -27,\label{ct}
\end{equation} 
called {\it Hesse's canonical equation} (see, e.g., \cite[Theorem 1.3.2.16]{Sc}). The associated form of $c_t$ is easily found to be
\begin{equation}
{\mathbf c}_t(z_1,z_2,z_3):=-\frac{1}{24(t^3+27)}(tz_1^3+tz_2^3+tz_3^3-18z_1z_2z_3).\label{bfct}
\end{equation}
For $t\ne 0$, $t^3\ne 216$ the cubic ${\mathbf c}_t$ is nondegenerate, and in this case the associated form of ${\mathbf c}_t$ is proportional to $c_t$, hence $(\PP\Phi)^2 (c_t) = c_t$.  Below we will see that the exceptional cubics $c_0$, $c_{6\tau}$, with $\tau^3=1$, are pairwise linearly equivalent.

It is well-known (see, e.g., \cite[Theorem 1.3.2.16]{Sc}) that $\PP\CC[z_1,z_2,z_3]_3^{\ss}$ is the union of $\PP X_3^3$ (which coincides with $\PP\CC[z_1,z_2,z_3]_3^{\s}$) and the following three orbits that consist of strictly semistable forms: ${\rm O}_1:=\SL_3\cdot\, z_1z_2z_3$, ${\rm O}_2:=\SL_3\cdot\, (z_1z_2z_3+z_3^3)$,\linebreak ${\rm O}_3:=\SL_3\cdot\, (z_1^3+z_1^2z_3+z_2^2z_3)$ (the cubics lying in ${\rm O}_3$ are called {\it nodal\,}). The dimensions of the orbits are 6, 7 and 8, respectively. Observe that ${\rm O}_1$ is closed in $\PP\CC[z_1,z_2,z_3]_3^{\ss}$ and is contained in the closures of each of ${\rm O}_2$, ${\rm O}_3$. We then see that $\PP\Phi$ maps $\PP X_3^3$ onto $\PP\CC[z_1,z_2,z_3]_3^{\ss}\setminus ({\rm O}_2\cup {\rm O}_3\cup {\rm O}_4)$, where ${\rm O}_4:=\SL_3\cdot\, c_0$ (as explained below, ${\rm O}_4$ also contains the other exceptional cubics $c_{6\tau}$, with\linebreak $\tau^3=1$). Further, note that the 8-dimensional orbit ${\rm O}_4$ is mapped by $\PP\Phi$ onto the 6-dimensional orbit ${\rm O}_1$ (thus the morphism of the stabilizers of $c_0$ and $\PP\Phi(c_0)$ is an inclusion of a finite group into a two-dimensional group). Hence, $\PP\Phi$ restricts to an equivariant involutive automorphism of $\PP X_3^3\setminus {\rm O}_4$, which for $t\ne 0$, $t^3\ne 216$ establishes a duality between the cubics $C c_t$ and $C^{-T} c_{-18/t}$ with $C\in\SL_3$, therefore between the orbits $\SL_3\cdot\, c_t$ and $\SL_3\cdot\, c_{-18/t}$.

To determine the induced map $\overline{\PP\Phi}$ of GIT quotients, we recall that the algebra of $\SL_3$-invariants $\CC[\CC[z_1,z_2,z_3]_3]^{\SL_3}$ is generated by the two {\it Aronhold invariants}\, $A_4$, $A_6$, of degrees 4 and 6, respectively. Explicit formulas for these invariants are given, e.g., in \cite[\S\S 220, 221]{Sal}, \cite{C}, and we recall that the expression for $A_4$ was written down in (\ref{aronhold4}). One has
\begin{equation}
\Delta=A_6^2+64\, A_4^3\label{discrtercub}
\end{equation}
(see \cite{C}), and for a ternary cubic of the form
\begin{equation}
f(z_1,z_2,z_3)=az_1^3+bz_2^3+cz_3^3+6dz_1z_2z_3\label{generaltercubic}
\end{equation}
the value of $A_6$ is calculated as 
\begin{equation}
\begin{array}{l}
A_6(f)=a^2b^2c^2-20abcd^3-8d^6.\label{form11}
\end{array}
\end{equation}
It then follows that the algebra $\CC[\PP X_3^3]^{\SL_3}\simeq\CC[X_3^3]^{\GL_3}$ is generated by
\begin{equation}
{\rm J}:=\frac{64A_4^3}{\Delta}.\label{form21}
\end{equation}
Hence, $\PP X_3^3\gitq\SL_3$ is the affine space $\CC$, and $\PP\CC[z_1,z_2,z_3]_3^{\ss}\gitq\SL_3$ is identified with $\PP^1$, where ${\rm O}_1$, ${\rm O}_2$, ${\rm O}_3$ project to the point at infinity in $\PP^1$.

Further, from formulas (\ref{aronhold4}), \eqref{ct}, (\ref{discrtercub}), \eqref{form11}, \eqref{form21} we find 
\begin{equation}
{\rm J}(c_t)=-\frac{t^3(t^3-216)^3}{1728(t^3+27)^3}\quad\hbox{for all $t$ with $t^3\ne -27$.}\label{form31}
\end{equation}
From identity \eqref{form31} one obtains
\begin{equation}
{\rm J}(c_0)={\rm J}(c_{6\tau})=0\quad\hbox{for $\tau^3=1$,}\label{Jeq11}
\end{equation}
which implies that the orbit ${\rm O}_4$ is given by the condition ${\rm J}=0$ and that the four cubics $c_0$, $c_{6\tau}$ are indeed pairwise linearly equivalent.

Using \eqref{bfct}, \eqref{form31} we see
$$
{\rm J}({\mathbf c}_t)=\frac{1}{{\rm J}(c_t)}\quad\hbox{for all $t\ne 0$ with $t^3\ne 216$.}\label{jtransftercubics}
$$
This shows that the map $\overline{\PP\Phi}$ extends to the involutive automorphism $\varphi$ of $\PP^1$ given by
$$
\zeta\mapsto\frac{1}{\zeta}.
$$
This involution preserves $\PP^1\setminus\{0,\infty\}$, which agrees with the duality between the orbits $\SL_3\cdot\, c_t$ and $\SL_3\cdot\, c_{-18/t}$ for $t\ne 0$, $t^3\ne 216$ established above. Next,\linebreak $\varphi(0)=\infty$, which corresponds to \eqref{Jeq11} and the facts that ${\rm O}_4$ is mapped onto ${\rm O}_1$. Also, one has $\varphi(\infty)=0$, but this identity cannot be illustrated by a correspondence between orbits.  Indeed, $\PP\Phi$ cannot be equivariantly extended to an involution of $\PP\CC[z_1,z_2,z_3]_3^{\ss}$ as  the fiber of the quotient $\PP\CC[z_1,z_2,z_3]_3^{\ss}\gitq\SL_2$ over $\infty$ contains ${\rm O}_1$, which cannot be mapped onto ${\rm O}_4$ since $\dim {\rm O}_1<\dim {\rm O}_4$.

Finally, an explicit calculation shows that $A_4({\mathbf c}_t)\ne 0$ for all $t^3\ne-27$ (cf.~Theorem \ref{ternarycubics}). Consider the absolute invariant of ternary cubics
$$
{\rm K}:=\frac{A_6^2}{64A_4^3}+1.
$$
It is then easy to see that ${\rm K}({\mathbf c}_t)=J(c_t)$ for all $t^3\ne-27$, which independently establishes Conjecture \ref{conj2} for $n=d=3$ (cf.~Corollary \ref{positivecor}).

\begin{remark}\label{Rem:binaryquarticternarycubics}
The above considerations easily imply Proposition \ref{ternarycubicsnabla}. Indeed, we see that $\im(\overline{\PP\Phi})=\varphi(\CC)=\PP^1\setminus\{0\}$ is a smooth curve. By factorization (\ref{decomposi1}) and Theorem \ref{psilocclosedimmer} it follows that $\im(\overline{\PP\nabla|_{X_3^3}})=(\overline{\PP\Psi})^{\,-1}(\PP^1\setminus\{0\})$ is a nonsingular curve as required.
\end{remark}

If we regard $\PP X_3^3$ as the space of elliptic curves, the invariant ${\rm J}$ of ternary cubics translates into the $j$-invariant, and one obtains an equivariant involution on the locus of elliptic curves with nonvanishing $j$-invariant. It is well-known that every elliptic curve can be realized as a double cover of $\PP^1$ branched over four points (see, e.g.,  \cite[p.~115, 117]{D1}, \cite[Exercise 22.37 and Proposition 22.38]{Harr}). Therefore, it is not surprising that the cases of binary quartics and ternary cubics considered above have many similarities.

\subsection{Rational equivariant involutions and projective duality}\label{uniqueness}
We have seen that the map $\overline{\PP\Phi}$ for binary quartics and ternary cubics yields involutions of $\PP^1$. It is natural to ask whether there exist any other involutions of $\PP^1$ that arise from rational equivariant involutions of $\PP\CC[z_1,z_2]_4$ and $\PP\CC[z_1,z_2,z_3]_3$ as above. Here for either $n=2$, $d=4$ or $n=d=3$ a rational map $\iota$ of $\PP\CC[z_1,\dots,z_n]_d$ is called equivariant if it satisfies
$$
\iota(Cf)=C^{-T}\iota(f),\,\,C\in\SL_n
$$
for all $f$ lying in the domain of $\iota$ (cf.~(\ref{equivartype2})). The following result asserts that there are no possibilities other than $\PP\Phi$:

\begin{theorem}\label{invclass}\cite[Theorem 2.1]{AIK} For each pair $n=2$, $d=4$ and $n=3$, $d=3$ the morphism $\PP\Phi$ is the unique rational equivariant involution of $\PP\CC[z_1,\dots,z_n]_d$.
\end{theorem}

We will now see that for $n=2$, $d=4$ and $n=d=3$ the unique rational equivariant involution of $\PP\CC[z_1,\dots,z_n]_d$, and therefore the orbit duality induced by $\PP\Phi$, can be understood via projective duality. We will now briefly recall this classical construction. For details the reader is referred to the comprehensive\linebreak survey \cite{T}. 

Let $V$ be a vector space. The dual projective space $(\PP V)^*$ is the algebraic variety of all hyperplanes in $V$, which is canonically isomorphic to $\PP V^*$. Let $X$ be a closed irreducible subvariety of $\PP V$ and $X_{\reg}$ the set of its regular points. Consider the affine cone $\hat X\subset V$ over $X$. For every $x\in X_{\reg}$ choose a point $\hat x\in\hat X$ lying over $x$. The cone $\hat X$ is regular at $\hat x$, and we consider the tangent space $T_{\hat x}(\hat X)$ to $\hat X$ at $\hat x$. Identifying $T_{\hat x}(\hat X)$ with a subspace of $V$, we now let $H_x$ be the collection of all hyperplanes in $V$ that contain $T_{\hat x}(\hat X)$ (clearly, this collection is independent of the choice of $\hat x$ over $x$). Regarding every hyperplane in $H_x$ as a point in $(\PP V)^*$, we obtain the subset
$$
{\mathcal H}:=\bigcup_{x\in X_{\reg}}H_x \subset (\PP V)^*.
$$
The Zariski closure $X^*$ of ${\mathcal H}$ in $(\PP V)^*$ is then called the variety dual to $X$. Canonically identifying $((\PP V)^*)^*$ with $\PP V$, one has the reflexivity property $X^{**}=X$. Furthermore, if $X$ is a hypersurface, there exists a natural map from $X_{\reg}$ to $X^*$, as follows:
$$
\varphi: X_{\reg}\to X^*,\quad x\mapsto T_{\hat x}(\hat X)\subset V,
$$
where $\hat x\in\hat X$ is related to $x\in X_{\reg}$ as above.  

Observe now that in each of the two cases $n=2$, $d=4$ and $n=d=3$, for $f\in \PP X_n^d$ the orbit $\SL_n\cdot\,f$ is a smooth irreducible hypersurface in $\PP X_n^d$, thus its closure $\overline{\SL_n\cdot\,f}$ in $\PP\CC[z_1,\dots,z_n]_d$ is an irreducible (possibly singular) hypersurface. Therefore, one can consider the map
$$
\varphi_f: \overline{\SL_n\cdot\,f}_{\reg}\to (\PP\CC[z_1,\dots,z_n]_d)^*\label{mapvarphismall}
$$
constructed as above. Then we have

\begin{theorem}\label{mainaik}\cite[Theorem 2.2]{AIK}  Suppose that we have either $n=2$, $d=4$, or $n=d=3$.  Then for every $f\in\PP X_n^d$ the restrictions $\PP\Phi\big|_{\SL_n\cdot\,f}$ and $\varphi_f\big|_{\SL_n\cdot\,f}$ coincide upon the canonical identification $(\PP\CC[z_1,\dots,z_n]_d)^*=\PP\CC[z_1,\dots,z_n]_d^*$ and the identification $\CC[z_1,\dots,z_n]_d^*=\CC[e_1,\dots,e_n]_d$ via the polar pairing.
\end{theorem}

This theorem provides a clear explanation of the duality for orbits of binary quartics and ternary cubics that we observed earlier in this section. Indeed, suppose first that $n=2$, $d=4$. Then Theorem \ref{mainaik} yields that for $t\ne 0,\pm 6$ one has $\overline{\SL_2\cdot \,q_t}^{\,*}\simeq\overline{\SL_2\cdot\, q_{-12/t}}$ and $\overline{O}_3^{\,*}\simeq\overline{O}_1$. By reflexivity it then follows that $\overline{O}_1^{\,*}\simeq\overline{O}_3$. However, since $O_1$ is not a hypersurface, there is no natural map from $\overline{O}_1$ to its dual. This fact corresponds to the impossibility to extend $\PP\Phi$ equivariantly to $O_1$.

Analogously, for $n=d=3$, Theorem \ref{mainaik} implies that for $t\ne 0$ and $t^3\ne 216$ we have $\overline{\SL_c\cdot\,c_t}^{\,*}\simeq\overline{\SL_3\cdot\,c_{-18/t}}$ and $\overline{{\rm O}}_4^{\,*}\simeq\overline{{\rm O}}_1$. By reflexivity one then has $\overline{{\rm O}}_1^{\,*}\simeq\overline{{\rm O}}_4$. Again, since ${\rm O}_1$ is not a hypersurface, there is no natural map from $\overline{{\rm O}}_1$ to its dual. This agrees with the nonexistence of an equivariant extension of $\PP\Phi$ to ${\rm O}_1$.

\section{Results and open problems concerning the contravariant arising from the morphism $\Phi$}\label{S:contravariant}
\setcounter{equation}{0}

\subsection{Covariants and contravariants} Recall that a regular function $\Gamma$ on the space $\CC[z_1,\dots,z_n]_k\times_{\CC}\CC^n$ (i.e., an element of $\CC[\CC[z_1,\dots,z_n]_k\times_{\CC}\CC^n]$) is said to be a {\it covariant}\, of forms in $\CC[z_1,\dots,z_n]_k$ if the following holds:
$$
\begin{array}{l}
\Gamma(f,z)=(\det C)^m\, \Gamma(C f,Cz)=(\det C)^m\, \Gamma(C f,z\, C^T),\\
\vspace{-0.3cm}\\
\hspace{3cm}f\in\CC[z_1,\dots,z_n]_k,\,\,z=(z_1,\dots,z_n)\in\CC^n,\,\,C\in\GL_n,
\end{array}
$$
where $m$ is an integer called the {\it weight} of $\Gamma$ and $z\mapsto Cz=z\, C^T$ is the standard action of $\GL_n$ on $\CC^n$ (see (\ref{actiononcn})). Every homogeneous component of\, $\Gamma$ with respect to $z$ is automatically homogeneous with respect to $f$ and is also a covariant. Such covariants are called {\it homogeneous}\, and their degrees with respect to $f$ and $z$ are called the {\it degree}\, and {\it order}, respectively. We may view a homogenous covariant $\Gamma$ of degree $D$ and order $K$ as the $\SL_n$-equivariant morphism
$$
\CC[z_1,\dots,z_n]_k  \to \CC[z_1,\dots,z_n]_K,\quad f  \mapsto (z, \mapsto \Gamma(f,z))
 $$
 of degree  $D$ with respect to $f$, which maps a form $f\in\CC[z_1,\dots,z_n]_k$ to the form in $\CC[z_1,\dots,z_n]_K$ whose evaluation at $z$ is $\Gamma(f,z)$. In what follows, we write $\Gamma(f)$ for the form  $z \mapsto \Gamma(f,z)$ on $\CC^n$. Covariants independent of $z$ (i.e., of order $0$) are called {\it relative invariants}. Note, for example, that the discriminant $\Delta$ is a relative invariant of forms in $\CC[z_1,\dots,z_n]_k$ of weight $k(k-1)^{n-1}$ hence of degree $n(d-1)^{n-1}$ (see \cite[Chapter 13]{GKZ}).

Next, we identify every element $z^*\in\CC^{n*}$ with its coordinate vector $(z_1^*,\dots,z_n^*)$ with respect to the basis $z_1,\dots,z_n$ of $\CC^{n*}$ and recall that $z^*\mapsto Cz^*=z^*\, C^{-1}$ is the standard action of $\GL_n$ on $\CC^{n*}$ (see (\ref{actiononcn*})). Then a regular function $\Lambda$ on the space $\CC[z_1,\dots,z_n]_k\times_{\CC}\CC^{n*}$ (i.e., an element of $\CC[\CC[z_1,\dots,z_n]_k\times_{\CC}\CC^{n*}]$) is said to be a {\it contravariant}\, of forms in $\CC[z_1,\dots,z_n]_k$ if one has
$$
\begin{array}{l}
\Lambda(f,z^*)=(\det C)^m\, \Lambda(C f,Cz^*)=(\det C)^m\, \Lambda(C f,z^* C^{-1}),\\
\vspace{-0.3cm}\\
\hspace{3cm}f\in\CC[z_1,\dots,z_n]_k,\,\,z^*=(z_1^*,\dots,z_n^*)\in\CC^{n*},\,\,C\in\GL_n,
\end{array}
$$
where $m$ is a (nonnegative) integer called the {\it weight} of $\Lambda$. Again, every contravariant splits into a sum of homogeneous ones, and for a homogeneous contravariant its degrees with respect to $f$ and $z^*$ are called the {\it degree}\, and {\it class}, respectively. We may regard a homogenous contravariant $\Lambda$ of degree $D$ and class $K$ as the $\SL_n$-equivariant morphism
$$
\CC[z_1,\dots,z_n]_k  \to \CC[z_1^*,\dots,z_n^*]_K,\quad  f  \mapsto (z^* \mapsto \Lambda(f,z^*))
$$
of degree $D$ with respect to $f$. In what follows, we write $\Lambda(f)$ for the form\linebreak $z^* \mapsto \Lambda(f,z^*)$ on $\CC^{n*}$.

If $n=2$, every homogeneous contravariant $\Lambda$ yields a homogenous covariant $\hat{\Lambda}$ via the formula
\begin{equation}
\hat{\Lambda}(f)(z_1,z_2) := \Lambda(f)(-z_2, z_1),\,\,f\in\CC[z_1,z_2]_k,\,\, (z_1,z_2)\in\CC^2,\label{relcovcontrav}
\end{equation}
where $(-z_2, z_1)$ is viewed as a point in $\CC^{2*}$. Analogously, every homogeneous covariant $\Gamma$ gives rise to a homogenous contravariant $\tilde{\Gamma}$ via the formula
$$
\tilde{\Gamma}(f)(z_1^*,z_2^*) := \Gamma(f)(z_2^*, -z_1^*),\,\,f\in\CC[z_1,z_2]_k,\,\, (z_1^*,z_2^*)\in\CC^{2*},\label{relcovcontrav1}
$$
where $(z_2^*,-z_1^*)$ is regarded as a point in $\CC^{2}$. Under these correspondences the degree and order of a homogeneous covariant translate into the degree and class of the corresponding homogeneous contravariant and vice versa.

\subsection{The contravariant arising from the morphism $\Phi$}

As before, fix $d\ge 3$ and recall that $\Phi$ is a morphism
$$
\Phi \co X_n^d \to \CC[e_1,\dots,e_n]_{n(d-2)}
$$
defined on the locus $X_n^d$ of nondegenerate forms. From now on we identify the spaces $\CC[e_1,\dots,e_n]_{n(d-2)}$ and $\CC[z_1^*,\dots,z_n^*]_{n(d-2)}$ by identifying $e_j$ and $z_j^*$ and regard $\Phi$ as the morphism from $X_n^d$ to $\CC[z_1^*,\dots,z_n^*]_{n(d-2)}$ given by in formulas (\ref{assocformdef}), (\ref{assocformexpp}). The coefficients $\mu_{i_1, \ldots, i_n}$ that determine $\Phi$ (see (\ref{assocformexpppp})) are elements of the coordinate ring $\CC[X_n^d] = \CC[\CC[z_1,\dots,z_n]_d]_{\Delta}$, i.e., have the form (\ref{formulaformus}).  Let $p_{i_1,\dots,i_n}$ in formula (\ref{formulaformus}) be the minimal integer such that $\Delta^{p_{i_1,\dots,i_n}}\cdot\mu_{i_1,\dots,i_n}$ is a regular function on $\CC[z_1,\dots,z_n]_d$ and
$$
p:=\max\{p_{i_1,\dots,i_n}: i_1+\dots+i_n=n(d-2)\}.
$$
Then the product $\Delta^p \Phi$ is the morphism
$$
\Delta^p \Phi \co X_n^d \to \CC[z_1^*,\dots,z_n^*]_{n(d-2)}, \quad f \mapsto \Delta(f)^p \Phi(f),
$$
which extends to a morphism from $\CC[z_1,\dots,z_n]_d$ to $\CC[z_1^*,\dots,z_n^*]_{n(d-2)}$. We denote the extended map by the same symbol $\Delta^p \Phi$. 

Notice that by Proposition \ref{equivariance} the morphism 
$$
\Delta^p \Phi \co \CC[z_1,\dots,z_n]_d \to\CC[z_1^*,\dots,z_n^*]_{n(d-2)}
$$
is in fact a homogeneous contravariant of weight $pd(d-1)^{n-1}-2$.  Since the class of $\Delta^p \Phi$ is $n(d-2)$, it follows that its degree is equal to $np(d-1)^{n-1}-n$. Observe that $p>0$ as the weight and the degree of a contravariant are always nonnegative.

In the next subsection we will see that $\Delta^p \Phi$ can be expressed via known contravariants for certain small values of $n$ and $d$. However, it appears that in full generality (i.e., for all $n\ge 2$, $d\ge 3$) the contravariant $\Delta^p \Phi$ has not been discovered prior to our work \cite{AIK}, \cite{I3}.

The contravariant $\Delta^p \Phi$ is rather mysterious with even its most basic properties not having been understood so far. Indeed, the very first question that one encounters is:

\begin{openprob}\label{probdegreecontravariant}
\it Compute the integer $p$. 
\end{openprob}

We will now state what is known regarding this problem starting with the following theorem: 

\begin{theorem}\label{mainoldpaper}\cite{AIK}, \cite{I3}. 
One has
\begin{equation}
p\le\left[\frac{n^{n-2}}{(n-1)!}\right],\label{estim}
\end{equation}
where $[x]$ denotes the largest integer that is less than or equal to $x$. Hence the degree of $\Delta^p \Phi$ does not exceed $n[n^{n-2}/(n-1)!](d-1)^{n-1}-n$.
\end{theorem}

Observe that for $n=2,3$ upper bound (\ref{estim}) yields $p=1$. However, (\ref{estim}) is not sharp in general. In the two propositions below we focus on the cases $n=4$, $n=5$ and find that for sufficiently small values of $d$ estimate (\ref{estim}) can be improved. 

Indeed, if $n=4$ inequality (\ref{estim}) yields $p\le 2$, whereas in fact the following holds:

\begin{proposition}\label{n=4}\cite{I3}
For $n=4$ one has
$$
%\left\{
\begin{array}{ll}
p=1 & \hbox{if $3\le d\le 6$,}\\
\vspace{-0.3cm}\\
p\le 2 & \hbox{if $d\ge 7$.}
\end{array}
%\right.
$$
\end{proposition}

\noindent Next, for $n=5$ inequality (\ref{estim}) yields $p\le 5$, but there are in fact more precise bounds:

\begin{proposition}\label{n=5}\cite{I3}
For $n=5$ one has
$$
%\left\{
\begin{array}{ll}
p=1 & \hbox{if $d=3$,}\\
\vspace{-0.3cm}\\
p\le 2 & \hbox{if $d=4$,}\\
\vspace{-0.3cm}\\
p\le 3 & \hbox{if $5\le d\le 8$,}\\
\vspace{-0.3cm}\\
p\le 4 & \hbox{if $9\le d\le 50$,}\\
\vspace{-0.3cm}\\
p\le 5 & \hbox{if $d\ge 51$.}\\
\vspace{-0.3cm}\\
\end{array}
%\right.
$$
\end{proposition}

The method used in the proofs of Propositions \ref{n=4}, \ref{n=5} can be applied, in principle, to any $n\ge 2$. However, an analysis of this kind appears to be computationally quite challenging to perform in full generality, and we did not attempt to do so systematically. We only give a word of warning that, although one may get the impression that the method always yields that $p=1$ if $d=3$, this is in fact not the case as the example of $n=6$ shows. Indeed, for $n=6$, $d=3$ the approach utilized  in the proofs of Propositions \ref{n=4}, \ref{n=5} only leads to the bound $p\le 2$.

Following the above discussion, we state a subproblem of Open Problem \ref{probdegreecontravariant}:

\begin{openprob}\label{openprobp>1}
\it Is there an example with $p>1${\rm ?}
\end{openprob}

In the next subsection we will look at the contravariant $\Delta^p\Phi$ in three special cases: (i) $n=2, d=4$, (ii) $n=2$, $d=5$, (iii) $n=d=3$. Recall that, by Theorem \ref{mainoldpaper}, in each of these cases we have $p=1$.

\subsection{The contravariant $\Delta^p \Phi$ for small values of $n$ and $d$}\label{contravarsmallnd}\hfill

\subsubsection{Binary Quartics} Let first $n=2$, $d=4$. In this case $\Delta \Phi$ is a contravariant of weight 10, degree 4 and class 4.  We have the following identity of covariants of weight 6 (see (\ref{relcovcontrav})):
\begin{equation}
\widehat {\Delta \Phi}= \frac{1}{2^7 3^3}I_2 \Hess -  \frac{1}{2^4}\Cat{\mathbf{id}},\label{covar1}
\end{equation}
where $I_2$ the relative invariant of degree $2$ considered in Subsection \ref{S:binary-quartics}, and\linebreak ${\mathbf{id}}:f\mapsto f$ the identity covariant. To verify (\ref{covar1}), it suffices to check it for the quartics $q_t$ introduced in (\ref{qt}). For these quartics the validity of (\ref{covar1}) is a  consequence of formulas (\ref{bfqt})--(\ref{form1}). 

Observe that formula (\ref{covar1}) is not a result of mere guesswork; it follows naturally from the well-known explicit description of the algebra of covariants of binary quartics. Indeed, this algebra is generated by $I_2$, the catalecticant $\Cat$, the Hessian $\Hess$ (which has degree 2 and order 4), the identity covariant ${\mathbf{id}}$ (which has degree 1 and order 4), and one more covariant of degree 3 and order 6 (see \cite[\S 145]{Ell}). Therefore $\widehat {\Delta \Phi}$, being a covariant of degree 4 and order 4, is necessarily a linear combination of $I_2\Hess$ and $\Cat{\mathbf {id}}$. The coefficients in the linear combination can be determined by computing $\Delta \Phi$, $I_2\Hess$ and $\Cat{\mathbf {id}}$ for particular nondegenerate quartics of simple form.

Formula (\ref{covar1}) yields an expression for the morphism $\Phi$ via $I_2$, $\Cat$ and $\Hess$. Namely, for $f\in X_2^4$ we obtain
\begin{equation} \label{eqn-quartic}
\Phi(f)(z_1^*,z_2^*)=\frac{1}{\Delta}\left(\frac{1}{2^7 3^3}I_2(f) \Hess(f)(z_2^*,-z_1^*)-\frac{1}{2^4}\Cat(f) f(z_2^*,-z_1^*)\right).
\end{equation}
One might hope that formula (\ref{eqn-quartic}) provides an extension of $\PP\Phi$ beyond $\PP X_2^4$. However, for $f=z_1^2z_2^2$ the second factor in the right-hand side of (\ref{eqn-quartic}) vanishes, which agrees with the fact, explained in Subsection \ref{S:binary-quartics}, that $\PP\Phi$ does not have a natural continuation to the orbit $O_1=\SL_2\cdot\,z_1^2z_2^2$.

\subsubsection{Binary Quintics} Suppose next that $n=2$, $d=5$. In this case the calculations are significantly more involved, and we will only provide a brief account of the result. In this situation $\Delta \Phi$ is a contravariant of weight 18, degree 6 and class 6. A generic binary quintic $f \in\CC[z_1,z_2]_5$ is linearly equivalent to a quintic given by the {\it Sylvester canonical equation}
\begin{equation}
f = a X^5 + b Y^5 + c Z^5,\label{sylvcanform}
\end{equation}
where $X$, $Y$, $Z$ are linear forms satisfying $X+Y+Z=0$ (see, e.g., \cite[\S 205]{Ell}). The algebra of $\SL_2$-invariants of binary quintics is generated by relative invariants of degrees 4, 8, 12, 18 with a relation in degree 36, and the algebra of covariants is generated by 23 fundamental homogeneous covariants (see \cite{Sy}), which we will write as $C_{i,j}$ where $i$ is the degree and $j$ is the order. 

For $f\in\CC[z_1,z_2]_5$ given in the form (\ref{sylvcanform}) the covariants relevant to our calculations are computed as follows:
$$
\begin{array}{l}
C_{4,0}(f,z)=a^2b^2+b^2c^2+a^2c^2-2abc(a+b+c), \\
\vspace{-0.3cm}\\
C_{8,0}(f,z)=a^2b^2c^2(ab+ac+bc),\hspace{0.8cm}  C_{5,1}(f,z)=abc(bcX+acY+abZ), \\
\vspace{-0.3cm}\\
C_{2,2}(f,z)=abXY+acXZ+bcYZ, \hspace{0.3cm} C_{3,3}(f,z)=abcXYZ, \\
\vspace{-0.3cm}\\
C_{4,4}(f,z)=abc(aX^4+bY^4+cZ^4),\hspace{0.4cm}  C_{1,5}(f,z)=f(z) =  a X^5 + b Y^5 + c Z^5,\\
\vspace{-0.3cm}\\
\displaystyle C_{2,6}(f,z)=\frac{\Hess(f)(z)}{400}=abX^3Y^3+bcY^3Z^3+acX^3Z^3.
\end{array}
$$
For instance, the discriminant can be written as
$$
\Delta=C_{4,0}^2-128\,C_{8,0}.
$$

The vector space of covariants of degree 6 and order 6 has dimension 4 and is generated by the products 
$$
\hbox{$C_{4,0}C_{2,6}$, $C_{1,5}C_{5,1}$, $C_{3,3}^2$, $C_{2,2}^3$, $C_{2,2}C_{4,4}$}
$$
satisfying the relation
$$
C_{4,0}C_{2,6} - C_{1,5}C_{5,1} + 9 C_{3,3}^2 - C_{2,2}^3 + 2 C_{2,2}C_{4,4}=0.
$$
One can then explicitly compute
$$
\widehat{\Delta \Phi}=\frac{1}{20}C_{4,0}C_{2,6}-\frac{3}{50}C_{1,5}C_{5,1}+\frac{27}{10}C_{3,3}^2-\frac{1}{10}C_{2,2}^3.
$$

\subsubsection{Ternary cubics} Finally, we assume that $n=d=3$. In this case $\Delta \Phi$ is a contravariant of weight 10, degree 9 and class 3. Recall that the algebra of $\SL_3$-invariants of ternary cubics is freely generated by the relative invariants $A_4$, $A_6$ (the Aronhold invariants considered in Subsection \ref{S:cubics}), and the ring of contravariants is generated over the algebra of $\SL_3$-invariants by the Pippian $P$ of degree $3$ and class $3$, the Quippian $Q$ of degree $5$ and class $3$, the Clebsch transfer of the discriminant of degree $4$ and class $6$, and the Hermite contravariant of degree $12$ and class $9$ (see \cite{C}). For a ternary cubic of the form (\ref{generaltercubic}), the Pippian and Quippian are calculated as follows:
$$
\hspace{-0.1cm}\begin{array}{l}
P(f)(z_1^*,z_2^*,z_3^*)  = -d(bcz_1^{*3}+acz_2^{*3}+abz_3^{*3})-(abc-4d^3)z_1^*z_2^*z_3^*,\\
\vspace{-0.1cm}\\
Q(f)(z_1^*,z_2^*,z_3^*) = (abc-10d^3)(bcz_1^{*3}+acz_2^{*3}+abz_3^{*3})-6d^2(5abc+4d^3)z_1^*z_2^*z_3^*.
\end{array}
$$
Since any contravariant of degree 9 and class 3 is a linear combination of $A_6 P$ and $A_4 Q$, it is easy to compute
\begin{equation}
\Delta \Phi = -\frac{1}{36}A_6 P - \frac{1}{27}A_4 Q.\label{contravar3}
\end{equation}
The above expression can be verified directly by applying it to the cubics $c_t$ defined in (\ref{ct}) and using formulas (\ref{bfct}), (\ref{discrtercub}), (\ref{form11}). 

Identity (\ref{contravar3}) provides an expression for $\Phi$ in terms of $A_4$, $A_6$, $P$ and $Q$. Namely, on $X_3^3$ we have
\begin{equation}
\Phi = -\frac{1}{\Delta}\left(\frac{1}{36}A_6 P + \frac{1}{27}A_4 Q\right).\label{newexpr1}
\end{equation}
One might think that formula (\ref{newexpr1}) yields a continuation of $\PP\Phi$ beyond $\PP X_3^3$. However, for $f=z_1z_2z_3$ the second factor in the right-hand side of (\ref{newexpr1}) is zero, which illustrates the obstruction to extending $\PP\Phi$ to the orbit ${\rm O}_1=\SL_3\cdot\,z_1z_2z_3$ discussed in Subsection \ref{S:cubics}.

\end{document}